\documentclass{amsart}

\usepackage{amsfonts}
\usepackage{amsmath}
\usepackage{amssymb}
\usepackage{mathrsfs}
\usepackage{graphicx}
\usepackage{geometry}
\usepackage{diagrams}
\usepackage[hyperfootnotes=false]{hyperref}
\hypersetup{colorlinks=false, linkcolor=green, citecolor=red, hyperfootnotes=true}

\newtheorem{theorem}{Theorem}[section]

\newtheorem{corollary}[theorem]{Corollary}

\newtheorem{definition}[theorem]{Definition}
\newtheorem{example}[theorem]{Example}

\newtheorem{lemma}[theorem]{Lemma}

\newtheorem{proposition}[theorem]{Proposition}
\newtheorem{remark}[theorem]{Remark}

\begin{document}

\renewcommand{\thefootnote}{}





\makeatletter 
\newcommand*{\defeq}{\mathrel{\rlap{%
                     \raisebox{0.3ex}{$\m@th\cdot$}}%
                     \raisebox{-0.3ex}{$\m@th\cdot$}}%
                     =} 
\makeatother 

\baselineskip=13pt


\title{Discrete Homotopy Theory and \\
Critical Values of Metric Spaces}


\author{Jim Conant} 
\address{Jim Conant \\ Department of Mathematics, University of Tennessee, Knoxville, TN 37996}
\email{jconant@math.utk.edu}

\author{Victoria Curnutte}
\address{Victoria Curnutte \\ Department of Mathematics, Colorado College, Colorado Springs, CO 80903}
\email{Victoria.Curnutte@ColoradoCollege.edu}

\author{Corey Jones}
\address{Corey Jones \\ Department of Mathematics, University of North Carolina-Asheville, Asheville, NC 28804}
\email{cmjones@unca.edu}

\author{Conrad Plaut}
\address{Conrad Plaut \\ Department of Mathematics, University of Tennessee, Knoxville, TN 37996}
\email{cplaut@math.utk.edu}

\author{Kristen Pueschel}
\address{Kristen Pueschel \\ Department of Mathematics, Cornell University, Ithaca, NY 14853-4201}
\email{kp338@cornell.edu}

\author{Maria Walpole}
\address{Maria Walpole \\ Department of Mathematics, Syracuse University, Syracuse, NY 13244}
\email{mwalpole@syr.edu}

\author{Jay Wilkins}
\address{Jay Wilkins \\ Department of Mathematics, University of Connecticut, Storrs, CT 06029}
\email{leonard.wilkins@uconn.edu}







\begin{abstract}
Utilizing the discrete homotopy methods developed for uniform spaces by Berestovskii-Plaut, we define the critical spectrum $Cr(X)$ of a metric space, generalizing to the non-geodesic case the covering spectrum defined by Sormani-Wei and the homotopy critical spectrum defined by Plaut-Wilkins. If $X$ is geodesic, $Cr(X)$ is the same as the homotopy critical spectrum, which differs from the covering spectrum by a factor of $\frac{3}{2}$. The latter two spectra are known to be discrete for compact geodesic spaces, and correspond to the values at which certain special covering maps, called $\delta $-covers (Sormani-Wei) or $\varepsilon $-covers (Plaut-Wilkins), change equivalence type. In this paper we initiate the study of these ideas for non-geodesic spaces, motivated by the need to understand the extent to which the accompanying covering maps are topological invariants. We show that discreteness of the critical spectrum for general metric spaces can fail in several ways, which we classify. The \textquotedblleft newcomer\textquotedblright\ critical values for compact, non-geodesic spaces are completely determined by the homotopy critical values and \textit{refinement} critical values, the latter of which can, in many cases, be removed by changing the metric in a bi-Lipschitz way.


\end{abstract}

\maketitle


\tableofcontents

\pagestyle{myheadings}
\markboth{J. Conant, V. Curnutte, C. Jones, C. Plaut, K. Pueschel, M. Walpole, J. Wilkins}{Discrete Homotopy Theory and Critical Values of Metric Spaces}


\section{Introduction}

The \textit{covering spectrum} of a geodesic space, introduced by Christina Sormani and Guofang Wei (\cite{SW2}, \cite{SW4}), provides a way to understand the topology of the space at a given scale. Their construction uses a classical method of Spanier (\cite{S}) that produces a covering map for a given open cover; in particular, they use the covering map $\pi^{\delta }:\widetilde{X}^{\delta }\rightarrow X$ of a geodesic space $X$ corresponding to the open cover by $\delta $-balls. Such a cover is constructed for every $\delta >0$, yielding a parameterized collection $\{\widetilde{X}^{\delta }\}_{\delta >0}$. The covering spectrum consists of those $\delta >0$ such that for all $\delta ^{\prime }>\delta $, the
covering maps $\pi ^{\delta'}$ and $\pi ^{\delta }$ are not equivalent (see also Remark \ref{SWR} below). Sormani and Wei used these covering maps, and the covering spectrum, to obtain interesting results concerning limits of compact Riemannian manifolds with Ricci curvature uniformly bounded below.

Valera Berestovskii and Conrad Plaut independently developed a method to construct covering maps determined by geometric information - first for topological groups (\cite{BPTG}) and later for uniform spaces (\cite{BPUU}, \cite{PQ}), hence metric spaces. Unlike Spanier's construction, which uses classical (continuous) paths and homotopies, the Berestovskii-Plaut construction uses discrete chains and homotopies, which we will discuss in more detail in the following section. One consequence is that the construction works for an arbitrary connected metric space and yields, for each $\varepsilon >0$, a covering map $\varphi _{\varepsilon}:X_{\varepsilon }\rightarrow X$ and corresponding deck group $\pi
_{\varepsilon }(X)$, called the $\varepsilon $-cover and $\varepsilon$-group, respectively. Plaut and Wilkins show in \cite{PW2} that for compact geodesic spaces the Sormani-Wei cover $\pi ^{\delta }$ is equivalent to $\varphi _{\varepsilon }$ when $\varepsilon =\frac{2\delta }{3}$. At the same time, in \cite{PW1}, \cite{PW2}, and \cite{W3}, they give several topological and geometric applications of the discrete methods.

For example, in \cite{PW1} the authors strengthen a theorem of M. Gromov (\cite{G}) that the fundamental group of any compact Riemannian manifold of diameter at most $D$ has a set of generators $g_{1},...,g_{k}$ of length at most $2D$ and relators of the form $g_{i}g_{m}=g_{j}$. In fact, they obtain an explicit bound for the number $k$ of generators in terms of the number of equivalence classes of \textquotedblleft short loops\textquotedblright\ at every point and the number of balls required to cover a given semi-locally simply connected geodesic space. As a corollary they obtain a \textquotedblleft curvature free\textquotedblright\ fundamental group finiteness theorem (new even for Riemannian manifolds) that generalizes the finiteness theorems of Anderson (\cite{An}) and Shen-Wei (\cite{ShW}). In \cite{PW2}, they investigate the covers of Gromov-Hausdorff convergent sequences, where discrete methods are particularly useful due to the well-established principle that many properties of the limit of a convergent sequence of compact metric spaces $X_{n}$ can be formulated in terms of the limits of finite point sets (cf. Chapter 7 of \cite{BBI}). They establish a type of weak closure result for a collection of covers called \textit{circle covers} that generalize $\varepsilon $-covers, allowing them to complete the investigation of Sormani-Wei into the structure of limits of $\delta $-covers. In \cite{W3}, Wilkins employs these discrete methods to relate the existence of universal covers of Peano continua to the geometry and topology of various generalized fundamental groups. For example, he shows that such a space has a universal cover if and only if the revised and uniform fundamental groups are finitely presented, which, in turn, holds if and only both groups are discrete in a specific topology called the \textit{covering topology}.

In \cite{PW2}, Plaut and Wilkins ask to what degree these concepts are topological invariants. The present paper represents a start on that program by moving beyond the geodesic case. There are many reasons why one would want to pursue such a program, including the existence of new topological invariants. As an example, given a Peano continuum $X$, one can consider the collection of all covering maps of $X$ (with connected domain) that are $\varepsilon $-covers for some geodesic metric (or more general metric) on $X$. For instance, up to rescaling, there is only one geodesic metric on a topological circle, and the only $\varepsilon $-covers for that metric are the trivial cover and the universal cover; the $n$-fold covers of the circle are not $\varepsilon$-covers for the geodesic metric. We conjecture that the same is true for any metric on the circle (assuming the covers have connected domain). On the other hand, $\mathbb{RP}^{2}$ with an open disc removed is homeomorphic to the Moebius band, and so there are geodesic metrics on the Moebius band that Gromov-Hausdorff approximate $\mathbb{RP}^{2}$. It follows from the convergence results of \cite{PW2} that the 2-fold cover of the Moebius band occurs as an $\varepsilon $-cover for some geodesic metrics. In other words, one can distinguish between the circle and the Moebius band on this basis, despite the fact that the circle is a deformation retraction of the Moebius band. However, other than directly classifying all of the $\varepsilon$-covers in examples like the circle, at this point we have few good methods to show that a given cover is \textit{not} an $\varepsilon $-cover for some geodesic metric, and it may be easier to consider more broadly the collection of covering maps that are $\varepsilon $-covers for some compatible metric in general. In this case one has to consider the possibility that $X_{\varepsilon }$ is not connected, and also that the analog of the covering spectrum may not be discrete in ${\mathbb{R}}_{+}$. The latter possibility is one of the main complications of the non-geodesic case, and understanding it is a main thrust of this paper.

Note also that these discrete methods - unlike geodesic methods - can be applied to fractals with no rectifiable curves like the Sierpinski Gasket with its \textquotedblleft resistance metric,\textquotedblright\ which is not geodesic and different from the metric induced by the plane (cf. \cite{K} and \cite{S1}). One may also consider more exotic spaces such as metrized solenoids, and - as was already observed in \cite{BPUU} - the Topologist's Sine Curve. Finally, one should recall that Gromov-Hausdorff convergence of non-compact proper geodesic spaces is generally reduced to consideration of metric balls which, even in geodesic spaces, need not themselves be geodesic spaces when equipped with the subspace metric.

The results of this paper are built around a natural defintion of the \textit{critical spectrum} for an arbitrary metric space. The precise definition (Definition \ref{crit value def}) requires some technical background given in Section 2, but the critical spectrum of a metric space $X$ - denoted by $Cr(X)$ - can be intuitively thought of as the set of all positive real numbers $\varepsilon $ at which the equivalence type of $\varepsilon $-covers changes as $\varepsilon $ decreases toward $0$. This spectrum extends both the covering spectrum of Sormani-Wei and the homotopy critical spectrum of Plaut-Wilkins (\cite{PW1}) to the non-geodesic case. The latter spectra are known to be discrete in ${\mathbb{R}}_{+}$ for any compact geodesic metric (see Section 3), but for non-geodesic spaces the situation is more complicated. 

First of all, we have 

\begin{theorem}
There exist compact, connected metric spaces with critical spectra that are not discrete in ${\mathbb{R}}_{+}$.
\end{theorem}

\noindent In Section 3 we give a formal classification of the ways in which discreteness may fail. In Section 4, we give examples of many of the types of critical values deduced in Section 3 using spaces we call Rapunzel's Combs. To begin with, one may identify two types of fundamental critical values - homotopy and refinement critical values. The resulting spectra are denoted by $H(X)$ and $R(X)$, respectively. Homotopy critical values are defined for metric spaces in the same way that they are defined for geodesic spaces, and are more topologically significant than refinement critical values, although the latter do indicate a lack of connectivity of certain metric balls (Proposition \ref{local conn thm}). In particular, $R(X)$ is empty for geodesic spaces, and refinement critical values can often be removed by changing the metric:

\begin{theorem}
\label{eintrinsic metric} Let $(X,d)$ be a compact, connected metric space. Then for every $\varepsilon >0$ there is a metric $D_{\varepsilon }$ on $X$ such that
\begin{enumerate}
\item[\emph{1)}] $d\leq D_{\varepsilon}$, and $d(x,y)=D_{\varepsilon }(x,y)$ if and only if $d(x,y)<\varepsilon$.
\item[\emph{2)}] $d$ and $D_{\varepsilon}$ are bi-Lipschitz equivalent.
\item[\emph{3)}] $D_{\varepsilon}$ has no refinement critical values greater than $\varepsilon $.
\end{enumerate}
\end{theorem}

\noindent The metric $D_{\varepsilon }$ is called the \textit{$\varepsilon $-intrinsic metric induced by $d$}. Note that the only obstruction in Theorem \ref{eintrinsic metric} to possibly removing \textit{all} refinement critical values is that one may have $\inf R(X) = 0$ (see Theorem \ref{ulcc}).

Rapunzel's Combs show that not all critical values are refinement or homotopy critical values, but these two types determine the entire spectrum:

\begin{theorem}
\label{spec equals closure hrcv} Let $X$ be a compact, connected metric space. If $\overline{H(X)}$ denotes the closure of the homotopy critical spectrum $H(X)$ in $\mathbb{R}_{+}=(0,\infty)$, then 
\begin{equation*}
Cr(X)=\overline{H(X)}\cup R(X)\text{.}
\end{equation*}
\end{theorem}

Since a geodesic space $X$ has no refinement critical values, the discreteness of $H(X)$ in $\mathbb{R}_+$ when $X$ is compact immediately yields the following.

\begin{corollary}
\label{geodesic crs} If $X$ is geodesic, then $Cr(X)=\overline{H(X)}$. If $X$ is also compact, then $Cr(X)=H(X)$.
\end{corollary}

\noindent We are also able to classify isolated critical values:

\begin{theorem}
\label{characterize isolated cv} If $X$ is a compact metric space, and if $\varepsilon >0$ is an isolated critical value of $X$, then $\varepsilon $ is a homotopy or a refinement critical value.
\end{theorem}

Finally, it is interesting to consider the limit of the metrics $D_{\varepsilon }$ described in Theorem \ref{eintrinsic metric}, which are monotone increasing as $\varepsilon \searrow 0$. In general the limit need not be finite, or if it is finite, topologically equivalent to the original metric. However, we do have:

\begin{theorem}
\label{ulcc} Let $(X,d)$ be a compact, connected metric space. If $D_{0}$ is finite, then $(X,D_{0})$ has no refinement critical values. If, in addition, $(X,d)$ and $(X,D_{0})$ are topologically equivalent then $D_{0}$ is a geodesic metric--and hence $X$ is locally path connected.
\end{theorem}

\begin{remark}
Some of the research contained in this paper was carried out by the undergraduate authors during an REU at the University of Tennessee in summer 2009 and during the subsequent year. The quality of contributions by undergraduate co-authors in published papers is known to vary from paper to paper. In the case of the present work, the contributions of the undergraduate authors are substantial, including the discovery of Rapunzel's Combs and insights into the distinct types of critical values in general metric spaces.
\end{remark}

\section{Background on Discrete Homotopy Theory}

Omitted proofs of statements in this section may be found in \cite{BPUU} or \cite{W2}. Given a metric space $(X,d)$, we will denote the open metric balls of radius $\varepsilon$ centered at $x\in X$ by $B(x,\varepsilon):=\{y \in X:d(x,y)<\varepsilon \}$. If multiple metrics on the same space are being considered, we will distinguish them by subscripts: $B_{d_1}(x,r)$, $B_{d_2}(x,r)$, etc. A map $f:Y\rightarrow X$ between metric spaces is an \textit{isometry} if $f$ is surjective and preserves distances (i.e. $d_{X}(f(y_{1}),f(y_{2}))=d_{Y}(y_{1},y_{2})$ for all $y_{1},y_{2}\in Y$). If $f:Y\rightarrow X$ is surjective and for some $\varepsilon >0$ the restriction of $f$ to any $B(y,\varepsilon )\subset Y$ is an isometry onto $B(f(y),\varepsilon )\subset X$, we call $f$ an $\varepsilon$-local isometry. It follows easily that an $\varepsilon$-local isometry is in fact a covering map such that $\varepsilon$-balls are evenly covered; when $\varepsilon$ is unimportant we may simply call $f$ a \textit{metric covering map}. Recall that $X$ is a \textit{geodesic space} if every pair of points in $X$ is joined by a \textit{geodesic}, namely an arclength parameterized curve having length equal to the distance between its endpoints. Note that when $X$ is a geodesic space, $Y$ is a connected topological space, and $f:Y\rightarrow X$ is any covering map, the well-known \textquotedblleft lifted metric\textquotedblright\ on $Y$ is the unique, topologically equivalent, geodesic metric on $Y$ such that $f$ is a metric covering map.

An $\varepsilon $\textit{-chain} $\alpha $ in $X$ is a finite sequence of points $\alpha =\{x_{0},x_{1},\dots ,x_{n-1},x_{n}\}$ such that $d(x_{i-1},x_{i}) <\varepsilon $ for $i=1,\dots ,n$. For any $\varepsilon $-chain $\alpha =\{x_{0},\dots,x_{n}\}$, the \textit{reversal} of $\alpha $ is denoted by $\alpha^{-1}:=\{x_{n},x_{n-1},\dots,x_{1},x_{0}\}$. A \textit{basic move} on an $\varepsilon$-chain is the addition or removal of a single point with the conditions that the endpoints remain fixed and the resulting chain is still an $\varepsilon$-chain. If two $\varepsilon$-chains $\alpha$ and $\beta$ have the same initial and terminal points, then we say they are $\varepsilon$\textit{-homotopic} if there is a finite sequence of $\varepsilon $-chains, $H=\{\alpha =\gamma _{0},\gamma _{1},\dots,\gamma_{k-1},\gamma _{k}=\beta \}$ - called an $\varepsilon $-homotopy - such that each $\gamma _{i}$ differs from $\gamma _{i-1}$ by a basic move. For a fixed basepoint $\ast \in X$, let $X_{\varepsilon }$ be the set of all $\varepsilon $-equivalence classes $[\{\ast =x_{0},\dots,x_{n}\}]_{\varepsilon }$ of $\varepsilon$-chains in $X$ beginning at $\ast$, which has a natural metric (Definition \ref{covermetric}). One may then define $\varphi _{\varepsilon}:X_{\varepsilon}\rightarrow X$ to be the \textquotedblleft endpoint map\textquotedblright\ taking $[\{\ast=x_{0},\dots ,x_{n}\}]_{\varepsilon }$ to $x_{n}$. If one assumes that $X$ is \textit{chain connected }in the sense that every pair of points is joined by $\varepsilon $-chains in $X$ for all $\varepsilon >0$, then $\varphi_{\varepsilon }:X_{\varepsilon }\rightarrow X$ is surjective, and in fact is a covering map and local isometry. Given $0<\delta <\varepsilon $, there is a well-defined map (cf. \cite{BPUU}) $\varphi _{\varepsilon \delta}:X_{\delta }\rightarrow X_{\varepsilon }$ that simply treats a $\delta $-chain like an $\varepsilon $-chain, i.e. $\varphi _{\varepsilon \delta}([\alpha ]_{\delta })=[\alpha ]_{\varepsilon }$. These maps satisfy the composition relation $\varphi_{\varepsilon \delta} =\varphi_{\varepsilon \lambda} \circ \varphi_{\lambda \delta}$ for $\delta <\lambda < \varepsilon$.

Basic arguments show that connected spaces are chain connected, and compact chain connected spaces are connected - although, for example, the rationals with the subspace metric from $\mathbb{R}$ are chain connected but not connected. To ensure that the maps $\varphi_{\varepsilon}$ are always covering maps, we assume from now on that all metric spaces are connected unless otherwise noted. We are concerned primarily with compact connected metric spaces.

Although it is important to keep in mind that all chains technically have both initial and terminal points, for simplicity we will often use the notation $\{x\}$ instead of $\{x,x\}$ for constant or trivial chains. The notation $\alpha \sim _{\varepsilon }\beta $ will denote the (equivalence) relation \textquotedblleft $\alpha $ is $\varepsilon $-homotopic to $\beta $.\textquotedblright\ We will often write $[x_{0},\dots,x_{n}]_{\varepsilon }$ rather than the technically correct $[\{x_{0},\dots,x_{n}\}]_{\varepsilon }$. There is a well-defined, associative concatenation operation on equivalence classes of $\varepsilon $-chains: if the initial point of $\beta$ equals the terminal point of $\alpha $, then we can define $[\alpha]_{\varepsilon }[\beta ]_{\varepsilon }=[\alpha \beta ]_{\varepsilon }$. The following properties are easily checked:

\begin{lemma}
Let $\alpha$ and $\beta$ be $\varepsilon$-chains. If the initial point of $\beta$ is the terminal point of $\alpha$, and if $\alpha^{\prime}\sim_{\varepsilon}\alpha$ and $\beta^{\prime}\sim_{\varepsilon}\beta$, then $\alpha\beta\sim_{\varepsilon}\alpha^{\prime}\beta^{\prime}$. Furthermore, $\alpha \alpha^{-1}$ is $\varepsilon$-null, and if $\alpha\sim_{\varepsilon}\beta$, then $\alpha^{-1}\sim_{\varepsilon}\beta^{-1}$.
\end{lemma}

An $\varepsilon $-\textit{loop} is an $\varepsilon $-chain with equal initial and terminal points. If an $\varepsilon $-loop $\alpha =\{\ast=x_{0},\dots ,x_{n}=\ast \}$ is $\varepsilon $-homotopic to the trivial loop $\{\ast \}$ then we say that $\alpha $ is $\varepsilon$-\textit{null}. If $\alpha $ is an $\varepsilon $-chain, $\delta <\varepsilon $, and $\alpha$ is $\varepsilon $-homotopic to a $\delta $-chain $\beta $, then we say that $\alpha$ \textit{can be} $\varepsilon$\textit{-refined to} $\beta$. We call such a homotopy an $\varepsilon $\textit{-refinement}, and refer to $\beta $ as a $\delta $-\textit{\textit{refinement}} of $\alpha $. When no confusion will result, we will often drop the $\varepsilon $'s and $\delta $'s in our notation. Note that if $\alpha $ is an $\varepsilon $-chain in a \textit{geodesic} space then we can refine $\alpha $ to any desired degree of fineness by successively adding midpoints of geodesics joining each pair of consecutive points in the chain. Refinements of this sort play important roles in \cite{PW1} and \cite{PW2}, and the ability (or inability) to refine chains is an important possible feature of non-geodesic spaces.

The set of classes $[\lambda ]_{\varepsilon }$, where $\lambda $ is an $\varepsilon $-loop at $\ast $, forms a group $\pi _{\varepsilon }(X)$ with operation induced by concatenation. This group acts on $X_{\varepsilon }$ by preconcatenation: for any $\varepsilon $-loop $\lambda $, define $h_{\lambda}([\alpha ]_{\varepsilon }):=[\lambda \alpha ]_{\varepsilon }$. Then $h_{\lambda }$ is a bijection that depends only on the $\varepsilon $-equivalence class of $\lambda $ and, as was shown in \cite{BPUU}, $\varphi_{\varepsilon }$ is a well-defined regular covering map with deck group naturally identified with $\pi _{\varepsilon }(X)$. One can think of $\pi_{\varepsilon }(X)$, roughly, as a fundamental group at a specific metric scale. For instance, for a geodesic space $X$ there is always a natural surjective homomorphism $h:\pi _{1}(X)\rightarrow \pi _{\varepsilon }(X)$, the kernel of which is generated by classes of loops that are, roughly, \textquotedblleft small\textquotedblright\ on the scale of $\varepsilon $.

We define a metric $d_{\varepsilon}$ on $X_{\varepsilon}$ as follows:

\begin{definition}
\label{covermetric} The length of an $\varepsilon$-chain, $\alpha=\{x_{0},\dots,x_{n}\}$, in $(X,d)$ is
\begin{equation*}
L(\alpha):=\sum_{i=1}^{n}d(x_{i-1},x_{i})\text{.}
\end{equation*}
The length of the equivalence class $[\alpha]_{\varepsilon}\in X_{\varepsilon}$ is $L([\alpha]_{\varepsilon}):=\inf\{L(\beta):\beta\in \lbrack\alpha]_{\varepsilon }\}$, and for $[\alpha]_{\varepsilon},[\beta]_{\varepsilon} \in X_{\varepsilon}$, we define $d_{\varepsilon}\bigl(\lbrack \alpha]_{\varepsilon},[\beta]_{\varepsilon}\bigr)=L\bigl(\lbrack \alpha^{-1}\beta]_{\varepsilon}\bigr)$.
\end{definition}

\noindent Note that if the initial point of $\beta $ is the terminal point of $\alpha $, then $L(\alpha \beta )=L(\alpha )+L(\beta )$. We also have $L(\alpha)=L(\alpha ^{-1})$ and, hence, $L([\alpha^{-1}]_{\varepsilon})=L([\alpha]_{\varepsilon })$.

It is straightforward to check that $L([\alpha \beta ]_{\varepsilon })\leq L([\alpha ]_{\varepsilon })+L([\beta ]_{\varepsilon })$, from which the triangle inequality for $d_{\varepsilon }$ follows. Symmetry is obvious, and positive definiteness is implied by the otherwise useful fact that if $L([\alpha ]_{\varepsilon })<\varepsilon $, where $\alpha =\{x_{0},\dots,x_{n}\}$ is an $\varepsilon $-chain, then $d(x_{0},x_{n})<\varepsilon $, $\alpha $ is $\varepsilon $-homotopic to the two-point $\varepsilon $-chain $\{x_{0},x_{n}\}$, and $L([\alpha]_{\varepsilon })=d(x_{0},x_{n})$.\ More generally one also sees that for $0<\delta \leq \varepsilon $ and $\varepsilon $-chain $\alpha =\{\ast =x_{0},\dots,x_{n}\}$, $B([\alpha]_{\varepsilon },\delta )\subset X_{\varepsilon }$ consists precisely of those $[\beta ]_{\varepsilon }\in X_{\varepsilon }$ such that $\beta $ is $\varepsilon $-homotopic to an $\varepsilon $-chain of the form $\{\ast=x_{0},\dots ,x_{n},y\}$, where $d(y,x_{n})<\delta $. It also holds that $\varphi _{\varepsilon }$ is $1$-Lipschitz (or distance non-increasing) and an $\frac{\varepsilon}{2}$-local isometry. In fact, every $\frac{\varepsilon}{2}$-ball in $X$ is evenly covered by a union of $\frac{\varepsilon}{2}$-balls in $X_{\varepsilon}$. As with the case of classical covering space theory, change of basepoint in a chain connected space induces a natural equivalence of covering spaces, which with the present metric is an isometry. Therefore we will treat base points very informally, but assuming, when needed, that maps are base point preserving. In particular, if $\ast$ is the base point in $X$ we will always use $[\ast ]_{\varepsilon }$ as the base point in $X_{\varepsilon}$. For brevity we will denote $[\ast]_{\varepsilon}$ by $\widetilde{\ast}$. Finally, the maps $h_{\lambda}:X_{\varepsilon}\rightarrow X_{\varepsilon}$ defining the action of $\pi _{\varepsilon}(X)$ on $X_{\varepsilon}$ are isometries.

There is an identification of fundamental importance, $\iota_{\delta\varepsilon }:X_{\delta }\rightarrow (X_{\varepsilon })_{\delta }$, that is defined by 
\begin{equation*}
\iota _{\delta \varepsilon }\Bigl(\lbrack \ast =x_{0},x_{1},\dots,x_{n}]_{\delta }\Bigr)=\Bigl[\lbrack \ast ]_{\varepsilon},[\ast ,x_{1}]_{\varepsilon },\dots,[\ast ,x_{1},\dots,x_{n}]_{\varepsilon } \Bigr]_{\delta }
\end{equation*}
\noindent and was shown in \cite{BPUU} to be a well-defined uniform homeomorphism. According to the metrics we have defined in the spaces $X_{\varepsilon }$ and $(X_{\varepsilon })_{\delta }$,
\begin{equation*}
d_{\varepsilon}\bigl(\lbrack \ast ,\dots,x_{i}]_{\varepsilon},[\ast,\dots,x_{i},x_{i+1}]_{\varepsilon }\bigr)=d(x_{i},x_{i+1}),
\end{equation*}
and, therefore, 
\begin{equation*}
L\bigl(\{\ast =x_{0},\dots,x_{n}\}\bigr)=L\Bigl(\bigl\{\lbrack\ast]_{\varepsilon },[\ast,x_{1}]_{\varepsilon},\dots,[\ast,x_{1},\dots,x_{n}]_{\varepsilon }\bigr\}\Bigr).
\end{equation*}
\noindent This implies that $\iota _{\delta \varepsilon }$ is in fact an isometry in the case of metric spaces. For one thing, this fact implies that the basic results that we have proved about $\varphi _{\delta}$ hold equally well for the function $\varphi _{\varepsilon \delta}$, including the properties of being Lipschitz and a $\frac{\delta}{2}$-local isometry when it is surjective.

We will denote the restriction of $\varphi _{\varepsilon \delta }$ to $\pi_{\delta }(X)$ by $\Phi _{\varepsilon \delta }:\pi_{\delta}(X)\rightarrow \pi _{\varepsilon }(X)$. These maps are homomorphisms (again, invoke $\iota_{\delta \varepsilon }$), and $\varphi_{\varepsilon \delta }$ is injective (respectively, surjective) if and only if $\Phi_{\varepsilon \delta }$ is injective (respectively, surjective). In the case of geodesic spaces, the functions $\varphi_{\varepsilon \delta }$ are always surjective. In fact, when $X$ is geodesic, the metric $d_{\varepsilon}$ coincides with the lifted geodesic metric (Proposition 24, \cite{PW1}), so that $X_{\varepsilon}$ is path connected and, hence, chain connected. The identification $\iota_{\delta \varepsilon }$ then tells us that $\varphi_{\varepsilon \delta}$ is surjective.

An $\varepsilon$-loop of the form $\{x,y,z,x\}$ will be called an $\varepsilon$\textit{-triangle}. Note that an $\varepsilon$-triangle is necessarily $\varepsilon$-null. If $f:Y\rightarrow X$ is a map between metric spaces, a lift of a chain, $\alpha=\{x_{0},\dots,x_{n}\}$, from $X$ to a point $y\in f^{-1}(x_{0})$ is a chain $\tilde{\alpha}=\{y=\tilde{x}_{0},\dots ,\tilde{x}_{n}\}$ such that $f(\tilde{x}_{i})=x_{i}$ for each $i=0,1,\dots,n$.

The next lemma shows that $X_{\varepsilon}$ largely inherits the local topology and metric properties of $X$. The first part follows directly from the fact that $\varphi_{\varepsilon}:X_{\varepsilon}\rightarrow X$ is an isometry from any $\frac{\varepsilon}{2}$-ball onto its image. The last parts were proved in \cite{BPUU} in the more general setting of uniform spaces.

\begin{lemma}
\label{ecover complete} Let $X$ be a connected metric space. If $X$ is locally compact \emph{(}respectively, complete\emph{)}, then $X_{\varepsilon }$ is locally compact \emph{(}respectively, complete\emph{)}. Furthermore, suppose the $\varepsilon$-balls in $X$ possess any one of the following properties: connected, chain connected, path connected. Then the whole space $X_{\varepsilon }$ is, respectively, connected, chain connected, path connected.
\end{lemma}

\begin{lemma}[Chain and Homotopy Lifting]
\label{chain lifting} Let $f:Y\rightarrow X$ be a surjective map between metric spaces that is a bijection from $\varepsilon $-balls in $Y$ onto $\varepsilon $-balls in $X$. Let $\alpha =\{x_{0},x_{1},\dots ,x_{n}\}$ be an $\varepsilon $-chain in $X$, and let $\tilde{x}_{0}$ be any point in $f^{-1}(x_0)$. Then $\alpha $ lifts uniquely to an $\varepsilon $-chain $\tilde{\alpha}$ beginning at $\tilde{x}_{0}$. If, in addition, $f$ has the property that the lift of any $\varepsilon $-triangle in $X$ is an $\varepsilon $-triangle in $Y$, and if $\beta $ is an $\varepsilon $-chain that begins at $x_{0}$ and is $\varepsilon $-homotopic to $\alpha $, then the lifts of $\alpha $ and $\beta $ to $\tilde{x}_{0}$ end at the same point and are $\varepsilon $-homotopic.
\end{lemma}

\begin{proof}
The first part is proved by induction on the number of points in $\alpha$. If $\alpha$ contains one point, the result is trivial. So, assume the result holds for all chains with $n$ or fewer points for some $n\geq1$. Let $\alpha=\{x_{0},\dots,x_{n}\}$ be an $\varepsilon$-chain with $n+1$ points, and let $\lambda=\{x_{0},\dots,x_{n-1}\}$. Using the inductive hypothesis, let $\tilde{\lambda}=\{\tilde{x}_{0},\dots,\tilde{x}_{n-1}\}$ be the unique lift of $\lambda$ to $\tilde{x}_{0}$. Then $f(\tilde{x}_{n-1})=x_{n-1}$, $f$ is a bijection from $B(\tilde{x}_{n-1},\varepsilon)$ onto $B(x_{n-1},\varepsilon)$, and $x_{n}\in B(x_{n-1},\varepsilon)$. Let $\tilde{x}_{n}$ be the unique point in $B(\tilde{x}_{n-1},\varepsilon)$ mapping to $x_{n}$ under $f$, and let $\tilde{\alpha}=\tilde{\lambda}\{\tilde{x}_{n}\}=\{\tilde{x}_{0},\dots,\tilde{x}_{n-1},\tilde{x}_{n}\}$. Then $\tilde{\alpha}$ is an $\varepsilon$-chain and $f(\tilde{\alpha})=\alpha$, proving existence. If there were another $\varepsilon$-chain, $\bar{\alpha}=\{\tilde{x}_{0}=\bar{x}_{0},\dots,\bar{x}_{n}\}$, beginning at $\tilde{x}_{0}$ and projecting to $\alpha$, then by the uniqueness part of the inductive hypothesis, the first $n$ points of $\bar{\alpha}$ and $\tilde{\lambda}$ must coincide: $\bar {x}_{i}=\tilde{x}_{i}$ for $0\leq i\leq n-1$. The bijectivity of $f$ on $B(\tilde{x}_{n-1},\varepsilon)$ then implies that $\tilde{x}_{n}$ and $\bar{x}_{n}$ must be the same, proving uniqueness.

To prove the second part, it suffices to consider the case in which $\alpha $ and $\beta $ differ by only a basic move, since an $\varepsilon$-homotopy is just a finite sequence of basic moves. Suppose $\beta $ is obtained by removing a point from $\alpha $, say $\alpha=\{x_{0},\dots,x_{i-1},x_{i},x_{i+1},\dots ,x_{n}\}$ and $\beta=\{x_{0},\dots,x_{i-1},x_{i+1},\dots ,x_{n}\}$. Let $\tilde{\alpha}$ and $\tilde{\beta}$ denote the unique lifts of $\alpha $ and $\beta $ to $\tilde{x}_{0}$. By uniqueness, $\tilde{\alpha}$ and $\tilde{\beta}$ must agree for their first $i$ points. Denote $\tilde{\alpha}$ by $\{\tilde{x}_{0},\dots ,\tilde{x}_{i-1},\tilde{x}_{i},\tilde{x}_{i+1},\dots ,\tilde{x}_{n}\}$ and $\tilde{\beta}$ by $\{\tilde{x_{0}},\dots ,\tilde{x}_{i-1},\tilde{y}_{i+1},\dots ,\tilde{y}_{n}\}$. Since we can remove $x_{i}$ from $\alpha $, the loop $\{x_{i-1},x_{i},x_{i+1},x_{i-1}\}$ is an $\varepsilon $-triangle. The lift of this $\varepsilon $-triangle to $\tilde{x}_{i-1}$ is an $\varepsilon $-triangle by hypothesis, and, by uniqueness of lifts, the first three points of that triangle must be $\tilde{x}_{i-1}$, $\tilde{x}_{i} $, and $\tilde{x}_{i+1}$. Since a triangle is a loop, the fourth point of this lift must be $\tilde{x}_{i-1}$. In other words, we have $d_{Y}(\tilde{x}_{i-1},\tilde{x}_{i+1})<\varepsilon $. So, $\tilde{x}_{i+1}$ and $\tilde{y}_{i+1}$ both lie in $B(\tilde{x}_{i-1},\varepsilon )$ and project under $f$ to $x_{i+1}$, implying that $\tilde{x}_{i+1}=\tilde{y}_{i+1}$. Finally, by uniqueness of lifts, the rest of $\tilde{\beta}$ agrees with $\tilde{\alpha}$. Thus, $\tilde{\beta}$ is obtained by removing a point from $\tilde{\alpha}$, so $\tilde{\alpha}\sim _{\varepsilon }\tilde{\beta}$. The case for the other basic move is similar.
\end{proof}

\begin{corollary}
\label{homotopy lifts} Let $X$ be a connected metric space, and let $\alpha $ and $\beta $ be $\varepsilon $-chains beginning at a common point $x\in X$ that are $\varepsilon $-homotopic. Then their lifts, $\tilde{\alpha}$ and $\tilde{\beta}$, to any $\tilde{x}\in \varphi _{\varepsilon }^{-1}(x)$ are $\varepsilon $-homotopic in $X_{\varepsilon }$.
\end{corollary}

\begin{proof}
Let $\{z_{0},z_{1},z_{2},z_{0}\}$ be an $\varepsilon $-triangle in $X$, and let $[\alpha ]_{\varepsilon }$ be any point in $\varphi_{\varepsilon}^{-1}(z_{0})$, where $\alpha=\{\ast =x_{0},\dots,x_{n}=z_{0}\}$. Let $\beta =\{\ast =x_{0},\dots ,x_{n}=z_{0},z_{1}\}$ and $\lambda =\{\ast=x_{0},\dots ,x_{n}=z_{0},z_{2}\}$. Then $d_{\varepsilon}([\alpha]_{\varepsilon },[\beta ]_{\varepsilon })$ and $d_{\varepsilon}([\alpha]_{\varepsilon },[\lambda ]_{\varepsilon })$ are less than $\varepsilon $. Moreover, $\beta ^{-1}\lambda \sim _{\varepsilon}\{z_{1},z_{0},z_{2}\}$, and since $d(z_{1},z_{2})<\varepsilon $, we can remove $z_{0}$ from this chain to conclude that $\beta^{-1}\lambda \sim_{\varepsilon}\{z_{1},z_{2}\}$. Thus, $d_{\varepsilon }([\beta]_{\varepsilon },[\lambda]_{\varepsilon})<\varepsilon $. It follows that $\{[\alpha ]_{\varepsilon },[\beta]_{\varepsilon },[\lambda ]_{\varepsilon},[\alpha ]_{\varepsilon }\}$ is an $\varepsilon $-triangle, and it projects under $\varphi_{\varepsilon}$ to $\{z_{0},z_{1},z_{2},z_{0}\}$. Uniqueness of lifts now implies that $\varepsilon $-triangles lift to $\varepsilon $-triangles, so Lemma \ref{chain lifting} then applies.
\end{proof}

\vspace{.1 in} For $\varepsilon$-covers, we can precisely characterize the lifts of $\varepsilon$-chains.

\begin{lemma}
\label{chain lifting ecover} Let $X$ be a connected metric space, and let $\varepsilon >0$ be given. If $\alpha =\{\ast =x_{0},\dots ,x_{n}\}$ is an $\varepsilon $-chain beginning at the base point $\ast \in X$ then the unique lift of $\alpha $ to $\tilde{\ast}=[\ast]_{\varepsilon }\in X_{\varepsilon }$ is given by 
\begin{equation*}
\tilde{\alpha}=\Bigl\{\lbrack \ast]_{\varepsilon},[x_{0},x_{1}]_{\varepsilon},\dots,[x_{0},\dots,x_{n-1}]_{\varepsilon },[x_{0},\dots,x_{n}]_{\varepsilon }= [\alpha]_{\varepsilon }\Bigr\}.
\end{equation*}
In particular, the endpoint of the lift of $\alpha $ is $[\alpha]_{\varepsilon }$ and the distances between consecutive points, as well as the chain length, are preserved in the lift.
\end{lemma}

\begin{proof}
First, note that, for $i=1,\dots,n$, 
\begin{equation*}
d_{\varepsilon}\bigl(\lbrack x_{0},\dots,x_{i-1}]_{\varepsilon},[x_{0},\dots,x_{i-1},x_{i}]_{\varepsilon}\bigr)=L\bigl(\lbrack x_{i-1},\dots,x_{1},x_{0},x_{1},\dots,x_{i-1},x_{i}]_{\varepsilon}\bigr)
\end{equation*}
\begin{equation*}
=L\bigl(\lbrack x_{i-1},x_{i}]_{\varepsilon}\bigr)=d(x_{i-1},x_{i})< \varepsilon\text{,}
\end{equation*}
since we can successively remove from $\{x_{i-1},\dots,x_1,x_0,x_1,\dots,x_{i-1},x_i\}$ the point $x_{0}$, then each $x_{1}$, and so on via an $\varepsilon$-homotopy. Thus, $\tilde{\alpha}$ is an $\varepsilon$-chain, and clearly $\varphi_{\varepsilon}(\tilde{\alpha})=\alpha$.
\end{proof}

\vspace{.1 in} \noindent The previous lemma immediately yields the following.

\begin{corollary}
\label{trivial loop lifts} If $X$ is a connected metric space and $\varepsilon >0$, then an $\varepsilon $-loop $\gamma $ based at $\ast$ lifts to an $\varepsilon $-loop at $\tilde{\ast}\in X_{\varepsilon }$ if and only if $\gamma $ is $\varepsilon $-null. Thus, any representative of a nontrivial element of $\pi _{\varepsilon }(X)$ lifts open \emph{(}i.e. to a non-loop\emph{)}.
\end{corollary}

\begin{lemma}
\label{Xe simply connected} For a connected metric space $X$ and any $\varepsilon >0$, $X_{\varepsilon}$ is $\varepsilon$-connected and $\varepsilon$-simply connected, i.e. every $\varepsilon $-loop based at $\tilde{\ast}\in X_{\varepsilon }$ is $\varepsilon$-null or, equivalently, $\pi _{\varepsilon }(X_{\varepsilon })$ is trivial.
\end{lemma}

\begin{proof}
The $\varepsilon$-connectivity follows from Lemma \ref{chain lifting ecover}. Given an $\varepsilon$-loop $\tilde{\gamma}$ at $\tilde{\ast}\in X_{\varepsilon}$, it will project to an $\varepsilon$-loop, $\gamma:=\varphi_{\varepsilon}(\tilde{\gamma})$, at $\ast$. Since $\gamma$ lifts to a closed loop, it is $\varepsilon$-null. By Corollary \ref{homotopy lifts}, this $\varepsilon$-nullhomotopy will lift to $X_{\varepsilon}$.
\end{proof}

We will postpone some concrete examples until the next section when we introduce the critical spectrum.

\begin{remark}
\label{SWR} In their definition of the covering spectrum, Sormani-Wei use the condition $\widetilde{X}^{\delta }\neq $ $\widetilde{X}^{\delta ^{\prime}}$ for all $\delta ^{\prime }>\delta $, but from their proofs it is clear that they take this to mean non-equivalence of the corresponding covering maps $\pi ^{\sigma }$ and $\pi ^{\delta }$. In fact, it seems to be an interesting open question (in our terminology) whether it is possible for $X_{\varepsilon }$ and $X_{\delta }$ to be homeomorphic when $\varphi_{\varepsilon }$ and $\varphi _{\delta }$ are not equivalent. Recall that it is possible in general for non-equivalent covers to involve homeomorphic spaces - for example with $n$-fold covers of the circle. However, these specific covers may not be $\varepsilon $-covers of the geodesic circle with any compatible metric, as we conjectured in the introduction.
\end{remark}

\section{The Critical Spectrum}

We begin by defining the two aforementioned fundamental types of critical values, the first of which was originally introduced in \cite{PW1}.

\begin{definition}
\label{homotopy cv def} A number $\varepsilon >0$ is a \textit{homotopy critical value} of $X$ if there is an $\varepsilon $-loop $\gamma $ in $X$ that is not $\varepsilon $-null but is $\delta$-null for all $\delta >\varepsilon $ (when $\gamma $ is considered as a $\delta $-chain). Such a $\gamma $ is an \textit{essential $\varepsilon $-loop}. The \textit{homotopy critical spectrum} of $X$ is the set $H(X)$ of all homotopy critial values.
\end{definition}

\begin{definition}
\label{homotopy refinement cv definition} Let $\alpha =\{x,y\}$ be a two-point chain in a metric space $X$ such that the following hold: \emph{1)} $d(x,y)=\varepsilon >0$; \emph{2)} for all $\delta $ greater than but sufficiently close to $\varepsilon $, $\alpha$ is not $\delta $-homotopic to an $\varepsilon $-chain. Then we say that $\varepsilon$ is a refinement critical value of $X$ and that $\{x,y\}$ is an essential $\varepsilon $-gap, or just an essential gap when $\varepsilon$ is clear. The set of all refinement critical values of $X$ is denoted by $R(X)$.
\end{definition}

Of course, since $d(x,y) = \varepsilon$, an essential $\varepsilon$-gap $\{x,y\}$ is not an $\varepsilon $-chain, and the point of the above definition is that, in a sense, we cannot \textquotedblleft refine\textquotedblright\ it to one either. The existence (or non-existence) of essential gaps is closely connected to the surjectivity of the maps $\varphi_{\varepsilon\delta }$; see Lemma \ref{characterize surjective}.

Even adding refinement critical values to the mix, however, does not tell the full story, and to completely understand the behavior of the $\varepsilon $-covers we need the following general definition of a critical value.

\begin{definition}
\label{crit value def} Let $X$ be a chain connected metric space. A non-critical interval of $X$ is a non-empty open interval $I\subset \mathbb{R}_{+}$ such that for each $\delta <\varepsilon $ in $I$, the map $\varphi_{\varepsilon \delta }:X_{\delta }\rightarrow X_{\varepsilon }$ is bijective. We call $\varepsilon > 0$ a critical value of $X$ if and only if
it does not lie in a non-critical interval. The set $Cr(X)$ of all critical values of $X$ is called the critical spectrum of $X$.
\end{definition}

\noindent There are two immediate consequences of the definition. First, $Cr(X)$ is bounded above by $diam(X)$ when $X$ is compact (or simply bounded). Second, the set of non-critical values, ${\mathbb{R}}_{+}\setminus Cr(X)$, is open in ${\mathbb{R}}_{+}$. This means that $Cr(X)$ is closed in $\mathbb{R}_+$, although $0$ may be a limit point of the critical spectrum.

Even though we are interested largely in non-geodesic spaces, it is useful to sketch out what happens in some simple but illustrative cases by starting with the geodesic case. Two non-geodesic examples follow the first two, and more are given in Section 4.

\begin{example}
\emph{Consider $X=S^1$, the geodesic circle of circumference $1$. The details of this example are given in \cite{W2}. If $\varepsilon > \frac{1}{3}$, then all $\varepsilon$-loops are $\varepsilon$-null, meaning that the $\varepsilon$-cover $X_{\varepsilon}$ is isometric to $X$ and is, therefore, the trivial cover. For any $0 < \varepsilon \leq \frac{1}{3}$, however, the $\varepsilon$-loop that wraps around the circle one time (in either direction) is no longer $\varepsilon$-null. The $\varepsilon$-cover for any $\varepsilon \leq \frac{1}{3}$ is now $\mathbb{R}$, the universal cover, and these covers ``unravel'' the hole in $X$. Since there is a $\frac{1}{3}$-loop in $X$ that is not $\frac{1}{3}$-null but is $\delta$-null for all $\delta > \frac{1}{3}$, it follows that $\frac{1}{3}$ is a homotopy critical value of $X$. It is, in fact, the only critical value of $X$, since $X_{\varepsilon}$ is universal for $0 < \varepsilon \leq \frac{1}{3}$. $\blacksquare$}
\end{example}

While basic, this first example illustrates some important phenomena. First, as noted in Corollary \ref{geodesic crs}, the critical spectrum of a compact geodesic space $X$ is simply its homotopy critical spectrum. Furthermore, as we mentioned previously, one of the fundamental results concerning the covering/homotopy critical spectra of a compact geodesic space $X$ is that this set is discrete in $\mathbb{R}_{+}:=(0,\infty )$ when $X$ is compact. These properties do not hold for general metric spaces. The discreteness was first proved by Sormani and Wei for their covering spectrum in \cite{SW2}, and a stronger result was later established by Plaut and Wilkins for the homotopy critical spectrum in \cite{PW1}. Thus, the situation in the previous example is, by extension, typical of compact geodesic spaces. One can imagine $\varepsilon >0$ sliding continuously along the positive real axis from the diameter of $X$ towards $0$ and consider the corresponding covers $\varphi_{\varepsilon}:X_{\varepsilon }\rightarrow X$. For $\varepsilon \geq diam(X)$, the covering map $\varphi _{\varepsilon }$ is trivial. As $\varepsilon $ decreases to $0$, the covering map remains trivial for a while until $\varepsilon $ is the first critical value of $X$, if one exists. At that point, the equivalence class of $\varphi_{\varepsilon}:X_{\varepsilon}\rightarrow X$ and the $\varepsilon$-group change from the trivial cover and group, respectively, to the next cover and group in the sequence. The covering map $\varphi _{\varepsilon }$ remains the same for some non-trivial interval, then changes again at the next critical value - again, if one exists.

The common discreteness of both the critical and covering spectra in the compact geodesic case is actually a consequence of a more general result. Plaut and Wilkins also showed in \cite{PW2} that despite the apparent difference in methods, the Berestovskii-Plaut construction, when applied to geodesic spaces, yields the same covers as the $\delta$-covers used by Sormani-Wei, for $\varepsilon =\frac{2\delta }{3}$. It follows that the covering spectrum and homotopy critical spectrum of a compact geodesic space differ only by a multiplicative factor of $\frac{2}{3}$. However, the requirement of chain connectedness is in general weaker than connectedness, and is even satisfied by some totally disconnected spaces. Therefore, the requirement of local path connectedness in the Spanier construction used by Sormani-Wei may be dropped when using discrete chains and homotopies. As was noted in the introduction, this opens the way to investigate all of these topics for non-geodesic spaces.

\begin{example}
\emph{For the geodesic Hawaiian Earring determined by circles with distinct circumferences $d_{i}\rightarrow 0$, the homotopy critical spectrum is $\{\frac{d_{1}}{3},\frac{d_{2}}{3},\dots\}$. At each of the critical values, another circle is \textquotedblleft unrolled,\textquotedblright\ and $X_{\varepsilon }$ consists of an infinite tree with successively smaller Hawaiian Earrings attached to each vertex. $\blacksquare$}
\end{example}

The previous example shows that for a compact geodesic space that is not semilocally simply connected, the critical spectrum is still discrete but may not have a positive lower bound. This is not true in general, though; there are compact geodesic spaces that are not semilocally simply connected but have finite or even empty critical spectra (cf. \cite{W3}). In general, for a compact, semilocally simply connected geodesic space, $X_{\varepsilon} $ is the universal cover of $X$ for all sufficiently small $\varepsilon >0$; in fact, a stronger statement holds in the more general setting of uniform spaces (Theorem 88, \cite{BPUU}). Note that in this case it is interesting to consider the so-called \textit{uniform universal cover of $X$} (\cite{BPUU}), which is the inverse limit of the covers $\varphi_{\varepsilon}:X_{\varepsilon }\rightarrow X$ as $\varepsilon \rightarrow 0$. The bonding maps are the functions $\varphi_{\varepsilon \delta }$ mentioned earlier. The uniform universal cover - which exists for all compact geodesic spaces, even if they are not semilocally simply connected - has the same lifting, regularity, and categorical universality properties as the traditional universal cover, though it is not generally a classical covering space defined by evenly covered neighborhoods. In this paper, however, we are only interested in the maps $\varphi _{\varepsilon }$ and $\varphi _{\delta\varepsilon }$ and not the inverse limit space.

The next two examples are non-geodesic, and Example \ref{geodesic circle with gap} gives a simple example of a refinement critical value.

\begin{example}
\label{euclidean square} \emph{Let $X$ be the square of side-length $s$ with its (non-geodesic!) Euclidean metric inherited from $\mathbb{R}^{2}$. It is not hard to see that, for $\varepsilon>s$, $X_{\varepsilon}$ is the trivial cover - while for $0<\varepsilon\leq s$, $X_{\varepsilon}$ is the universal cover, which is homeomorphic to $\mathbb{R}$. $\blacksquare$}
\end{example}

\begin{example}
\label{geodesic circle with gap} \emph{Let $S^{1}$ be the geodesic circle of circumference $1$. Fix $a,b\in S^{1}$ so that $d(a,b)=\frac{1}{4}$, and remove the open geodesic segment from $a$ to $b$. Let $X\subset S^1$ be the resulting set with the inherited subspace metric, not the induced geodesic metric, which would just make $X$ a line segment. Choose as the base point $* $ the midpoint of the longer segment from $a$ to $b$. See Figure \ref{fig:gap_circle} (showing the full sequence of $\varepsilon$-covers). As in the case of $S^{1}$, $X_{\varepsilon}$ is the trivial cover for $\varepsilon > \frac{1}{3}$. Since $\varepsilon $ is larger than the gap we created by removing the segment, an $\varepsilon$-loop can cross over the gap; informally speaking, $X_{\varepsilon}$ does not ``see'' the gap. For $\frac{1}{4} <\varepsilon \leq \frac{1}{3}$, one can show that $\pi_{\varepsilon}(X)\cong \mathbb{Z}$. In this case, $X_{\varepsilon}$ is the subspace of $\mathbb{R}$ with the open segments $(\frac{n}{2}-\frac{1}{8},\frac{n}{2}+\frac{1}{8})$, $n\in \mathbb{Z}$, removed. The $\varepsilon$-group, $\pi_{\varepsilon}(X)$, acts by shifts, as in the case of $S^{1}$ itself. Intuitively, we still unravel the circle as in the standard geodesic case, but the missing segment of the circle gets unraveled along with it. Note that $X_{\varepsilon }$ is not connected in this case.}

\begin{figure}[hbt]
\centering
\includegraphics[scale=1.0]{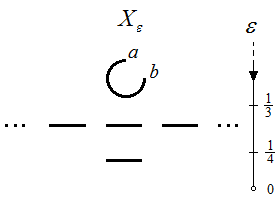}
\caption{The $\protect\varepsilon$-covers of the geodesic circle with a gap.}
\label{fig:gap_circle}
\end{figure}

\emph{Now, suppose $0<\varepsilon \leq \frac{1}{4}$. Since the removed segment has length $\frac{1}{4}$, it becomes the case at $\varepsilon =\frac{1}{4}$ that no $\varepsilon $-loop can cross the gap. Thus, the $\varepsilon$-covers now recognize the gap, and it is now impossible to travel around the circle via an $\varepsilon $-loop. In other words, all $\varepsilon $-loops are necessarily trivial, and $\pi _{\varepsilon }(X)$ is the trivial group. The pair $\{a,b\}$ is, in fact, an essential gap, and $\frac{1}{4}$ is a refinement critical value. In this case, $H(X)=\left\{\frac{1}{3}\right\}$, $R(X) = \left\{\frac{1}{4}\right\}$, and $Cr(X) = \left\{\frac{1}{4},\frac{1}{3}\right\}$.}

\emph{A subtle point deserves mention: $\varphi_{\varepsilon}:X_{\varepsilon}\rightarrow X$ is topologically trivial for $\varepsilon >\frac{1}{3}$ and $0<\varepsilon \leq \frac{1}{4}$, but it is only an isometry only for $\varepsilon >\frac{1}{3}$. The reason for this is that, for $\varepsilon\leq \frac{1}{4}$, the metric that we have defined makes $X_{\varepsilon }$ isometric to a Euclidean segment. In fact, the changes in the topology of the covering space are measuring geometric features of the space rather than topological features. This example illustrates two things that can occur in non-geodesic spaces but not in geodesic ones: the covering space $X_{\varepsilon }$ may not be connected even when $X$ is path connected and locally path connected, and $\pi_{\varepsilon }(X)$ may be non-trivial even when $\pi _{1}(X)$ is trivial. $\blacksquare$}
\end{example}

We now continue to analyze the general critical spectrum.

\begin{lemma}
\label{characterize surjective} Let $X$ be a metric space, and let $0<\delta <\varepsilon $ be given. The following are equivalent: \emph{1)} $\varphi_{\varepsilon \delta}:X_{\delta} \rightarrow X_{\varepsilon}$ is surjective; \emph{2)} every $\varepsilon $-chain in $X$ can be $\varepsilon $-refined to a $\delta $-chain; \emph{3)} every two-point $\varepsilon$-chain in $X$ can be $\varepsilon $-refined to a $\delta $-chain.
\end{lemma}

\begin{proof}
The equivalence of $2$ and $3$ is obvious, since an $\varepsilon$-chain is the concatenation of two-point $\varepsilon$-chains. The equivalence of 1 and parts 2 and 3 follows from the definition of $\varphi_{\varepsilon\delta}$. Specifically, given $[\alpha]_{\varepsilon} \in X_{\varepsilon}$, the existence of an element $[\beta]_{\delta} \in X_{\delta}$ such that $[\beta]_{\varepsilon} = \varphi_{\varepsilon \delta}([\beta]_{\delta}) = [\alpha]_{\varepsilon}$ means precisely that $\alpha$ is $\varepsilon$-homotopic to the $\delta$-chain $\beta$.
\end{proof}

In the following, we may occasionally need to refer to the bonding map between two covers $X_{\delta}$ and $X_{\varepsilon}$ when either $\delta<\varepsilon $ or vice versa is possible. In this case, we will simply refer to ``the map between $X_{\delta }$ and $X_{\varepsilon }$,'' and this will mean $\varphi_{\varepsilon \delta}$ if $\delta < \varepsilon$ and $\varphi_{\delta \varepsilon}$ if $\varepsilon < \delta$.

\begin{lemma}
\label{critical value seq} A positive number $\varepsilon $ is in $Cr(X)$ if and only if there is a sequence $\{\varepsilon _{n}\}$ such that $\varepsilon_{n}\neq \varepsilon$ for all $n$, $\varepsilon _{n}\rightarrow \varepsilon$, and the map between $X_{\varepsilon }$ and $X_{\varepsilon_{n}} $ is not bijective for all $n$.
\end{lemma}

\begin{proof}
If such a sequence exists, then $\varepsilon $ cannot lie in a non-critical interval, for if it did then $\varepsilon _{n}$ would also lie in that interval for all large $n$, contradicting that the map between $X_{\varepsilon}$ and $X_{\varepsilon_n}$ is not bijective.

Conversely, suppose $\varepsilon \in Cr(X)$. Then there is no non-critical interval containing $\varepsilon $. So there are two positive numbers, $r_{1} $ and $r_{2}$, in $(\varepsilon -1,\varepsilon +1)$ such that $r_{1}<r_{2}$ and the map $\varphi _{r_{2}r_{1}}:X_{r_{1}}\rightarrow X_{r_{2}}$ is not bijective. We will show that at least one of the maps between $X_{\varepsilon }$ and $X_{r_{1}}$, $X_{r_{2}}$ must be non-bijective. In other words, we can choose $\varepsilon _{1}$ to be $r_{1}$ or $r_{2}$, and the proof is complete by iterating this process to construct $\{\varepsilon_{n}\}$. If $\varepsilon $ equals $r_{1}$ or $r_{2}$, we are done. If $r_{1}<r_{2}<\varepsilon$, then $\varphi_{\varepsilon r_{1}}=\varphi_{\varepsilon r_{2}}\circ \varphi _{r_{2}r_{1}}$. If both $\varphi_{\varepsilon r_{1}}$ and $\varphi _{\varepsilon r_{2}}$ were bijective, then $\varphi _{r_{2}r_{1}}$ would be also, a contradiction. The proofs for the other two possible orderings are similar.
\end{proof}

\vspace{.1 in} In light of this result, there are four possible ways in which a critical value might occur. That is, $\varepsilon \in Cr(X)$ if and only if one or more of the following hold (the sloped arrows indicate strictly increasing or decreasing sequences):

\begin{enumerate}
\item[1)] There is a sequence, $\varepsilon _{n}\searrow \varepsilon $, such that the map $\varphi _{\varepsilon_{n}\varepsilon}:X_{\varepsilon}\rightarrow X_{\varepsilon _{n}}$ is non-injective (resp. non-surjective) for all $n$. In these cases, we say that $X$ is $\varepsilon$\textit{-upper non-injective} (resp. $\varepsilon$\textit{-upper non-surjective}), and we call $\varepsilon$ an \textit{upper non-injective critical value} (resp. \textit{upper non-surjective critical value}).

\item[2)] There is a sequence, $\varepsilon _{n}\nearrow \varepsilon $, such that the map $\varphi_{\varepsilon\varepsilon_n}:X_{\varepsilon_{n}}\rightarrow X_{\varepsilon }$ is non-injective (resp. non-surjective) for all $n$. In these cases, we say that $X$ is $\varepsilon$\textit{-lower non-injective} (resp. $\varepsilon $\textit{-lower non-surjective}), and we call $\varepsilon$ a \textit{lower non-injective critical value} (resp. \textit{lower non-surjective critical value}).
\end{enumerate}

\noindent As we will see in the last section, each of these possibilities may occur in a compact metric space, and, in fact, two or more of these cases may simultaneously hold for a given $\varepsilon$.

\begin{remark}
\label{refinable} Spaces for which all maps $\varphi _{\varepsilon \delta }$ are surjective exhibit a much stronger connection between critical values and topology. Such spaces - which include, for example, all geodesic spaces, all convex subsets of Euclidean space, and all Peano continua - are called refinable. By Definition \ref{homotopy refinement cv definition}, such spaces have no esssential gaps and, thus, no refinement critical values.
\end{remark}

If $\varepsilon $ is a homotopy critical value of $X$, then the maps $X_{\varepsilon }\rightarrow X_{\varepsilon +t}$ are non-injective for all $t>0$, meaning that $X$ is $\varepsilon$-upper non-injective. The converse of this statement is true when $X$ is compact geodesic; that is, every upper non-injective critical value is a homotopy critical value. If $\varepsilon$ is a refinement critical value of $X$ then the maps $X_{\varepsilon}\rightarrow X_{\varepsilon +t}$ are non-surjective for all sufficiently small $t$, and $\varepsilon $ is an upper non-surjective critical value. The converse of this statement is not true in general. Example \ref{geodesic circle with gap} shows that the two notions may coincide, but it may also be the case that $\varepsilon$ is an upper non-surjective critical value when there is no essential $\varepsilon $-gap, as illustrated in Examples \ref{noncompact non-gap} and \ref{rapunzel comb four}. These examples also show that the second part of Lemma \ref{indicating cv} below cannot be improved to the existence of a refinement critical value.

\begin{lemma}
\label{indicating cv} Let $0<\delta <\varepsilon $ be given. If $\varphi_{\varepsilon \delta }:X_{\delta }\rightarrow X_{\varepsilon }$ is not injective, then there is a homotopy critical value of $X$ in the interval $[\delta ,\varepsilon )$. If $\varphi _{\varepsilon\delta}:X_{\delta}\rightarrow X_{\varepsilon }$ is not surjective, then there is an upper non-surjective critical value in $[\delta ,\varepsilon )$.
\end{lemma}

\begin{proof}
If $\varphi _{\varepsilon \delta }:X_{\delta }\rightarrow X_{\varepsilon }$ is not injective, then there is a non-trivial $\delta $-loop $\gamma $ that is $\varepsilon $-null. Since an $\varepsilon $-chain is also an $(\varepsilon -t)$-chain for sufficiently small $t$, and since an $\varepsilon $-homotopy is just a finite sequence of $\varepsilon $-chains, it follows that an $\varepsilon $-homotopy is also an $(\varepsilon -t)$-homotopy for sufficiently small $t>0$. Hence, $\gamma$ is also $(\varepsilon -t)$-null for small $t$. Let $\varepsilon^{\ast }=\inf \{\tau \in (\delta ,\varepsilon ):\gamma \text{ is }\tau \text{-null}\}$. Note that $\varepsilon ^{\ast }$ cannot be in this set, for if $\gamma $ were $\varepsilon ^{\ast }$-null, it would be $(\varepsilon ^{\ast}-t)$-null for small $t$, contradicting that $\varepsilon ^{\ast }$ is the infimum. Thus $\gamma $ is $\varepsilon ^{\ast }$-nontrivial but $(\varepsilon ^{\ast }+t)$-null for all $t>0$, making $\varepsilon^{\ast }$ a homotopy critical value.

If $\varphi _{\varepsilon \delta }:X_{\delta }\rightarrow X_{\varepsilon }$ is not surjective then there is a two-point $\varepsilon $-chain $\gamma=\{x,y\}$ that cannot be $\varepsilon $-refined to a $\delta $-chain. Let $l=d(x,y)<\varepsilon $ and $\varepsilon ^{\ast }$ be the infimum of the set $R$ of all $\tau \in (\delta ,\varepsilon )$ such that $\gamma $ can be $\varepsilon $-refined to a $\tau $-chain. Clearly we have $\delta \leq \varepsilon ^{\ast }\leq l<\varepsilon $. Moreover, $\gamma $ cannot be $\varepsilon $-refined to an $\varepsilon ^{\ast }$-chain, for such a chain would also be an $(\varepsilon ^{\ast }-t)$-chain for sufficiently small $t>0 $, contradicting $\varepsilon ^{\ast }=\inf R$. Suppose $\varepsilon^{\ast } $ were not an upper non-surjective critical value. Then there would be some interval $[\varepsilon ^{\ast },\varepsilon ^{\ast }+t^{\ast })$ such that $\varepsilon^* + t^* < \varepsilon$ and, for every $0<t<t^{\ast }$, every $(\varepsilon ^{\ast }+t)$-chain can be $(\varepsilon ^{\ast }+t)$-refined to an $\varepsilon ^{\ast }$-chain. But by definition of $\varepsilon^{\ast }$, we can $\varepsilon $-refine $\gamma $ to an $(\varepsilon ^{\ast }+t)$-chain for some $t<t^{\ast }$. Then we could $(\varepsilon ^{\ast }+t)$-refine that chain to an $\varepsilon ^{\ast }$-chain, yielding an $\varepsilon $-refinement of $\gamma $ to an $\varepsilon ^{\ast }$-chain, a contradiction.
\end{proof}

\begin{lemma}
\label{lower non inject} For a connected metric space $X$ the following hold:
\begin{enumerate}
\item[\emph{1)}] $\varepsilon $ is a lower non-injective critical value of $X$ if and only if it is the upper limit of a strictly increasing sequence of homotopy critical values;
\item[\emph{2)}] if $\varepsilon$ is an upper non-injective critical value, then it is either a homotopy critical value or the limit of a strictly decreasing sequence of homotopy critical values.
\end{enumerate}
\end{lemma}

\begin{proof}
If $\varepsilon $ is a lower non-injective critical value of $X$ then there exists a sequence $\varepsilon _{n}\nearrow \varepsilon $ such that each map $X_{\varepsilon_{n}}\rightarrow X_{\varepsilon }$ is non-injective. By Lemma \ref{indicating cv}, for each $n$ there is a homotopy critical value $\varepsilon _{n}^{\ast }$ in the interval $[\varepsilon _{n},\varepsilon )$. From the sequence $\{\varepsilon _{n}^{\ast }\}$ we may choose any strictly increasing subsequence.

Conversely, suppose there is a sequence of homotopy critical values, $\{\varepsilon _{n}\}$, such that $\varepsilon _{n}\nearrow \varepsilon$. Then for each $n$, there is some $t>0$ such that $\varepsilon_{n}<\varepsilon _{n}+t<\varepsilon$ and the map $\varphi_{\varepsilon_{n}+t,\varepsilon _{n}}:X_{\varepsilon _{n}}\rightarrow X_{\varepsilon_{n}+t}$ is not injective. The map $\varphi _{\varepsilon\varepsilon_{n}}:X_{\varepsilon_{n}}\rightarrow X_{\varepsilon }$ is equal to the composition $\varphi_{\varepsilon ,\varepsilon _{n}+t}\circ \varphi_{\varepsilon_{n}+t,\varepsilon _{n}}$. Since $\varphi_{\varepsilon_{n}+t,\varepsilon_{n}}$ is not injective, neither is $\varphi_{\varepsilon \varepsilon _{n}}$, completing the proof of the first statement.

For the second statement suppose $\varepsilon $ is not a homotopy critical value and consider a sequence $\varepsilon _{n}\searrow \varepsilon $ such that each map $\varphi_{\varepsilon_{n}\varepsilon}:X_{\varepsilon}\rightarrow X_{\varepsilon _{n}}$ is non-injective. Then there is a homotopy critical value $\varepsilon_{n}^{\ast }$ in $[\varepsilon ,\varepsilon _{n})$, and since $\varepsilon_{n}^{\ast }\neq \varepsilon $, we may choose any strictly decreasing subsequence of $\{\varepsilon _{n}^{\ast }\}$.
\end{proof}

\begin{lemma}
\label{lower non surj} Let $X$ be a metric space and $\varepsilon >0$ a lower non-surjective critical value. Then there exists a sequence of upper non-surjective critical values converging up to $\varepsilon $. If in addition $X$ is compact, then $\varepsilon $ is also an upper non-injective critical value.
\end{lemma}

\begin{proof}
The proof of the first part is an application of Lemma \ref{indicating cv}, similar to part 1 of Lemma \ref{lower non inject}. Going further, by hypothesis, there is an $\varepsilon $-chain $\alpha_{1}=\{x_{1},y_{1}\}$ that cannot be $\varepsilon $-refined to an $\frac{\varepsilon }{2}$-chain. Let $\varepsilon_{1}=d(x_{1},y_{1})$, and note that $\varepsilon>\varepsilon _{1}\geq \frac{\varepsilon }{2}$. Now, let $\tau _{1}=\max\{\varepsilon _{1},\varepsilon -\frac{\varepsilon }{3}\}$ and $\tau_{1}^{\ast }=\frac{\tau_1 + \varepsilon}{2}$. Then let $\alpha_{2}=\{x_{2},y_{2}\}$ be an $\varepsilon$-chain that cannot be $\varepsilon$-refined to a $\tau_{1}^{\ast }$-chain. Let $\varepsilon_{2}=d(x_{2},y_{2})$, noting that $\varepsilon _{2}\geq \tau _{1}^{\ast }>\tau _{1}\geq \varepsilon _{1}$ and $\varepsilon >\varepsilon _{2}>\varepsilon -\frac{\varepsilon }{3}$. Continuing this process inductively, we obtain a sequence of $\varepsilon $-chains, $\alpha _{n}=\{x_{n},y_{n}\}$ with $\varepsilon_{n}:=d(x_{n},y_{n})$ such that 1) $\varepsilon -\frac{\varepsilon }{n+1} <\varepsilon _{n}<\varepsilon$ for each $n\geq 2$, and 2) for each $n\geq 1$, there is some $\tau $ strictly between $\varepsilon_{n}<\varepsilon _{n+1}$ such that $\alpha _{n+1}$ cannot be $\varepsilon $-refined to a $\tau $-chain.\ In particular, $\varepsilon_{n}\nearrow\varepsilon $.

If $X$ is compact, by choosing subsequences if necessary we may suppose that: $x_{n}\rightarrow x$ and $y_{n}\rightarrow y$ for some $x,y$ such that $d(x,y)=\varepsilon $, and $d(x_{n},x_{n-1})$, $d(y_{n},y_{n-1})<\frac{\varepsilon }{2}$ for all $n$.

For each $n$, let $\beta _{n} := \{x_{n},x_{n-1},y_{n-1},y_{n}\}$, and let $\gamma _{n}$ denote the loop $\beta_{n}\alpha_{n}^{-1}=\{x_{n},x_{n-1},y_{n-1},y_{n},x_{n}\}$. Note that $\beta_n$ is an $\varepsilon$-chain and $\gamma_n$ is an $\varepsilon$-loop. For any $\lambda>\varepsilon _{n-1}$, $\beta _{n}$ is a $\lambda $-chain, since $d(x_{n},x_{n-1})$, $d(y_{n},y_{n-1})<\frac{\varepsilon }{2} \leq \varepsilon_{n-1}$ and $d(x_{n-1},y_{n-1})=\varepsilon _{n-1}$. In addition, the chain $\alpha _{n}$ cannot be $\varepsilon$-homotopic to $\beta _{n}$; if this were true, then $\alpha _{n}$ would be $\varepsilon $-homotopic to a chain in which all distances are at most $\varepsilon_{n-1}$, contradicting the fact that there is some $\tau $ strictly between $\varepsilon _{n-1}$ and $\varepsilon _{n}$ such that $\alpha _{n}$ is not $\varepsilon $-homotopic to a $\tau $-chain. It follows that $\gamma _{n}$ is not $\varepsilon $-null. We next observe that that $\varepsilon_{n-1}+d(x_{n},x_{n-1})\geq \varepsilon $ and $\varepsilon_{n-1}+d(y_{n},y_{n-1})\geq \varepsilon $. In fact, if the first inequality did not hold, we would have the following $\varepsilon $-nullhomotopy between $\gamma_n$ and $\{x_n\}$:
\begin{eqnarray}
\gamma _{n}=\{x_{n},x_{n-1},y_{n-1},y_{n},x_{n}\} &\sim& \{x_{n},y_{n-1},y_{n},x_{n}\}  \label{nullhomotopy} \\
&\sim& \{x_{n},y_{n},x_{n}\}  \notag \\
&\sim& \{x_{n},x_{n}\},  \notag
\end{eqnarray}
\noindent where the first step is allowed because 
\begin{equation*}
d(x_{n},y_{n-1})\leq d(x_{n},x_{n-1})+d(x_{n-1},y_{n-1})=d(x_{n},x_{n-1})+\varepsilon_{n-1}< \varepsilon.
\end{equation*}
\noindent A similar conclusion follows if the other inequality does not hold. Finally, define 
\begin{equation*}
\delta_{n}:=\max\{\varepsilon_{n-1}+d(x_{n},x_{n-1}),\varepsilon_{n-1}+d(y_{n},y_{n-1})\} \geq \varepsilon.
\end{equation*}
\noindent Then $\gamma _{n}$ is $(\delta _{n}+\frac{1}{n})$-null, via the same sequence of steps in homotopy \ref{nullhomotopy}. This means that the map $X_{\varepsilon }\rightarrow X_{\delta _{n}+\frac{1}{n}}$ is not injective. Since $\delta _{n}+\frac{1}{n}\rightarrow \varepsilon $, it now follows that $\varepsilon $ is an upper non-injective critical value.
\end{proof}

\begin{lemma}
\label{refex}Suppose that $X$ is a compact metric space, $\varepsilon $ is an upper non-surjective critical value, and there exists $\delta>\varepsilon$ such that whenever $\varepsilon <\delta _{1}<\delta_{2}<\delta $, the map $\varphi _{\delta _{2}\delta _{1}}$ is injective. Then $\varepsilon $ is a refinement critical value.
\end{lemma}

\begin{proof}
By assumption, there is a sequence $\varepsilon _{n}\searrow \varepsilon $, with $\varepsilon _{n}<\delta $ for all $n$, such that each map $X_{\varepsilon }\rightarrow X_{\varepsilon _{n}}$ is not surjective. That is, for each $n$ there is an $\varepsilon _{n}$-chain $\gamma_{n}=\{x_{n},y_{n}\}$ that cannot be $\varepsilon _{n}$-refined to an $\varepsilon $-chain. We claim that for any $n \geq 1$ $\gamma _{n}$ cannot be $\varepsilon _{1}$ -refined to an $\varepsilon $-chain either. If it could be then, equivalently, $[\gamma _{n}]_{\varepsilon _{1}}$ would lie in the image of $\varphi_{\varepsilon _{1}\varepsilon }$. But by assumpion, $\varphi _{\varepsilon_{1}\varepsilon _{n}}$ is an injection, and since $\varphi _{\varepsilon_{1}\varepsilon }=\varphi _{\varepsilon_{1}\varepsilon_{n}}\circ \varphi_{\varepsilon _{n}\varepsilon }$, that would place $[\gamma_{n}]_{\varepsilon _{n}}$ in the image of $\varphi_{\varepsilon_{n}\varepsilon }$, a contradiction.

Since $X$ is compact, by choosing subsequences if necessary we may suppose $x_{n}\rightarrow x$ and $y_{n}\rightarrow y$ for some $x,y$ with $d(x,y)=\varepsilon $. We finish by proving that $\gamma:=\{x,y\}$ is an essential gap. Suppose, to the contrary, that we can find values $\delta$ strictly greater than but arbitrarily close to $\varepsilon $ such that $\gamma$ can be $\delta $-refined to an $\varepsilon $-chain. We may choose $\delta<\varepsilon _{1}$ and $n$ large enough that $d(x_{n},y_{n})+d(x_{n},x)+d(y_{n},y)<\delta $ and $d(x_n,x)$, $d(y_n,y) < \varepsilon$. Noting that $d(x,y_{n})\leq d(x,x_{n})+d(x_{n},y_{n})<\delta $, there is a $\delta $-homotopy 
\begin{equation*}
\gamma _{n}=\{x_{n},y_{n}\}\rightarrow \{x_{n},x,y_{n}\}\rightarrow \{x_{n},x,y,y_{n}\}.
\end{equation*}
\noindent Combining this with a $\delta $-refinement of $\{x,y\}$ to an $\varepsilon $-chain, we obtain a $\delta $-refinement of $\gamma _{n}$ to an $\varepsilon $-chain. Since $\delta <\varepsilon _{1}$, this $\delta$-refinement would also be an $\varepsilon _{1}$-refinement, which we showed previously does not exist.
\end{proof}

\vspace{.1 in}

\begin{proof}[Proof of Theorem \protect\ref{characterize isolated cv}]
Lemmas \ref{lower non inject} and \ref{lower non surj} together imply that $\varepsilon $ must be upper non-injective or upper non-surjective. If the former holds, then Lemma \ref{lower non inject} and the fact that this critical value is isolated imply that $\varepsilon $ is a homotopy critical value. If $\varepsilon $ is upper non-surjective then since $\varepsilon $ is an isolated critical value, the hypotheses of Lemma \ref{refex} are satisfied and $\varepsilon $ is a refinement critical value.
\end{proof}

\vspace{0.1in}Compactness is only required for the second part of the proof of Theorem \ref{characterize isolated cv}, but the theorem is false without it:

\begin{example}
\label{noncompact non-gap} \emph{Let $X$ be the union of the graphs of $f(x)=1+e^{x}$ and $g(x)=-1-e^{x}$, for $x\leq 0$, and the vertical segment connecting the right endpoints of each graph. Give $X$ the subspace metric inherited from $\mathbb{R}^{2}$. Pictured as a subset of $\mathbb{R}^{2}$, as $x\rightarrow -\infty $ the tails of this space asymptotically approach a distance of $2$ from each other. In fact, $2$ is an upper non-surjective critical value. To see this, take any pair of vertically aligned points, $\{(x,-1-e^{x}),(x,1+e^{x})\}$, for any $x << 0$. The distance between these points is $2+2e^{x}$. For all $\varepsilon $ greater than but sufficiently close to this value, this chain can be $\varepsilon $-refined to a $2+\tau $ chain for any $\tau >0$, but it cannot be refined to a $2$-chain. (If $\varepsilon $ is large enough we can simply refine it around the other end.\emph{)} However, there is no essential $2$-gap in this space. $\blacksquare$}
\end{example}

\begin{proof}[Proof of Theorem \protect\ref{spec equals closure hrcv}]
The containment $\overline{H(X)}\cup R(X)\subset Cr(X)$ is clear since $Cr(X) $ is closed in $\mathbb{R}_+$. Conversely, let $\varepsilon >0$ be a critical value of $X$. Lemma \ref{lower non inject} immediately handles the cases when $\varepsilon$ is upper or lower non-injective. If $\varepsilon $ is lower non-surjective then we may use Lemma \ref{lower non surj} followed by Lemma \ref{lower non inject}. If $\varepsilon $ is upper non-surjective then by Lemma \ref{refex} we need only consider the following case: there exist, for each $n\geq 1$, $\tau _{n}$ and $\delta _{n}$ such that $\varepsilon <\delta _{n}<\tau_{n}<\varepsilon +\frac{1}{n}$ and the map $\varphi _{\tau _{n}\delta_{n}}:X_{\delta_{n}}\rightarrow X_{\tau _{n}}$ is non-injective. But then Lemma \ref{indicating cv} implies there is a homotopy critical value in each interval $[\delta_{n},\tau _{n})$, and the proof is finished.
\end{proof}

\vspace{.1 in} The next few results further examine the utility and topological significance of refinement critical values.

\begin{proposition}
\label{local conn thm} If $\{x,y\}$ is an essential $\varepsilon$-gap in a metric space $X$, then the balls $B(x,\delta)$ and $B(y,\delta)$ are disconnected for all $\delta$ greater than but sufficiently close to $\varepsilon$.
\end{proposition}

\begin{proof}
Let $\alpha := \{x,y\}$ be an essential $\varepsilon$-gap in $X$. There is a $\tau > \varepsilon$ such that for all $\delta \in (\varepsilon,\tau)$, $\alpha$ is not $\delta$-homotopic to an $\varepsilon$-chain. We claim that $B(x,\delta)$ is not connected for all such $\delta$; a parallel argument holds for $B(y,\delta)$. In fact, $B(x,\delta)$ is not even $\varepsilon$-connected. If $B(x,\delta)$ were $\varepsilon$-connected, there would be an $\varepsilon$-chain $\beta = \{x=x_0,x_1,\dots,x_n=y\}$ lying in $B(x,\delta) $. Then we could construct a $\delta$-homotopy between $\beta$ and $\alpha$ by just successively removing $x_1$, $x_2,\dots, x_{n-1}$. But this contradicts that $\{x,y\}$ is an essential $\varepsilon$-gap.
\end{proof}

\vspace{.1 in} \noindent The topologist's sine curve with its Euclidean metric has no refinement critical values but has a continuum of points at which small balls are not connected. Thus, the converse of the previous proposition is not true.

\begin{definition}
\label{eintrinsic metric def} Let $(X,d)$ be a connected metric space with $\varepsilon > 0$. The $\varepsilon$-intrinsic metric determined or induced by $d$ is the metric $D_{\varepsilon}$ defined by 
\begin{equation*}
D_{\varepsilon}(x,y) = \inf \{ L(\alpha) : \alpha \ \text{is an $\varepsilon$-chain in $(X,d)$ from $x$ to $y$}\}.
\end{equation*}
\end{definition}

The fact that $D_{\varepsilon}$ is a metric follows from the same type of argument that shows $d_{\varepsilon}$ is a metric on the $\varepsilon$-cover, and it follows immediately from the triangle inequality that $d \leq D_{\varepsilon}$. This inequality and the fact that $\{x,y\}$ is the shortest $\varepsilon$-chain between $x$ and $y$ when $d(x,y) < \varepsilon$ further imply that $D_{\varepsilon}(x,y) < \varepsilon$ if and only if $d(x,y) < \varepsilon$, in which case the two metrics agree. This also shows that not only are $(X,d)$ and $(X,D_{\varepsilon})$ locally isometric but the topologies induced by $d$ and $D_{\varepsilon}$ are equivalent. Finally, for fixed $x,y \in X$, it is easy to see that $D_{\varepsilon}(x,y) \leq D_{\delta}(x,y)$ for $\delta < \varepsilon$, since any $\delta$-chain from $x$ to $y$ is also an $\varepsilon$-chain. Thus, the function $(\varepsilon,x,y) \mapsto D_{\varepsilon}(x,y)$ is monotone decreasing in $\varepsilon$, or increasing as $\varepsilon \rightarrow 0$.

The $\varepsilon$-intrinsic metric induced by a given metric can be thought of as a discretized or coarse analog of an induced length metric. If $(X,d)$ is already a length or geodesic space, then $D_{\varepsilon}$ is equal to $d$ for every $\varepsilon$. Thus, one can roughly think of the difference between a given metric $d$ and the induced metric $D_{\varepsilon}$ as a measure - at a particular metric scale - of how far $d$ is from being intrinsic in some sense. It is also interesting to note that the $\varepsilon $-cover, $X_{\varepsilon}$, with its natural metric $d_{\varepsilon}$, is \textit{always} $\varepsilon$-intrinsic, whether $(X,d)$ is geodesic or not.

One of the results of Theorem \ref{eintrinsic metric} is that, for each $\varepsilon >0$, one can replace the given metric $d$ on a compact, connected metric space $X$ with a topologically equivalent metric that eliminates all refinement critical values greater than $\varepsilon $. In particular, we have the following immediate corollary.

\begin{corollary}
\label{eliminate rcv} If $(X,d)$ is a compact, connected metric space such that $\inf R(X) > 0$, then there is a metric $\bar{d}$ on $X$ that is topologically equivalent to $d$ and is such that $(X,\bar{d})$ has no refinement critical values.
\end{corollary}

\noindent This result also supports the idea put forth in the introduction that refinement critical values arise not as much as a result of the underlying topology of the given space as they do as a result of the particular metric one imposes on the space.

\begin{proof}[Proof of Theorem \protect\ref{eintrinsic metric}]
Part 1 has already been noted. Consider the metric space $(X,D_{\varepsilon}) $. We claim that for any $\delta > \varepsilon$, any two-point $\delta$-chain can be $\delta$-refined to an $\varepsilon$-chain, showing that there are no refinement critical values greater than $\varepsilon$. But this essentially follows from the definition of $D_{\varepsilon}$. If $D_{\varepsilon}(x,y) < \delta$, then there is - with respect to $d$ - an $\varepsilon$-chain $\alpha =\{x=x_0,x_1,\dots,x_n=y\}$ such that $L(\alpha) = \sum_{i=1}^n d(x_{i-1},x_i) < \delta$. Since the metrics agree on $\varepsilon$-balls, $\alpha$ is also an $\varepsilon$-chain in $(X,D_{\varepsilon})$. To see that $\alpha$ is $\delta$-homotopic to $\{x,y\}$ in $(X,D_{\varepsilon})$, we construct a $\delta$-homotopy as follows. Since $\{x_0,x_1,x_2\}$ is an $\varepsilon$-chain with respect to $d$, we have $D_{\varepsilon}(x_0,x_2) \leq d(x_0,x_1) + d(x_1,x_2) \leq \sum_{i=1}^n d(x_{i-1},x_i) < \delta$. Thus, we can remove $x_1$ from $\alpha$ via $\delta $-homotopy to obtain $\{x_0,x_2,\dots,x_n\}$. Similarly, $\{x_0,x_1,x_2,x_3\} $ is an $\varepsilon$-chain in $(X,d)$ from $x_0$ to $x_3$, and $D_{\varepsilon}(x_0,x_3) \leq \sum_{i=1}^3 d(x_{i-1},x_i) \leq \sum_{i=1}^n d(x_{i-1},x_i) < \delta$. So, we can then remove $x_2$ from the previous chain to obtain $\{x_0,x_3,\dots,x_n\}$. Continuing in the obvious way, we obtain a $\delta$-homotopy in $(X,D_{\varepsilon})$ from $\alpha$ to $\{x = x_0,x_n=y\}$.

Now, assume $(X,d)$ is compact, which implies that there is a natural number $M$ such that any two points in $(X,d)$ can be joined by an $\varepsilon$-chain having at most $M$ points. If $x$, $y \in X$ and $d(x,y) < \varepsilon$, then $D_{\varepsilon}(x,y) = d(x,y)$ and we trivially have $D_{\varepsilon}(x,y) \leq Md(x,y)$. Suppose $d(x,y) \geq \varepsilon$, and let $\alpha$ be an $\varepsilon$-chain in $(X,d)$ from $x$ to $y$ with $M$ or fewer points. It follows that $D_{\varepsilon}(x,y) \leq L(\alpha) \leq M\varepsilon \leq Md(x,y)$. Thus, $D_{\varepsilon} \leq Md$, and the two metrics are bi-Lipschitz equivalent.
\end{proof}

\vspace{.1 in} Since $D_{\varepsilon}(x,y)$ is monotone increasing as $\varepsilon$ decreases, it is natural to consider what happens when we let $\varepsilon$ go to $0$. If we set 
\begin{equation*}
D_0(x,y) = \sup_{\varepsilon} D_{\varepsilon}(x,y),
\end{equation*}
then it is easy to see that $D_0$ is also a metric on $X$ when $D_0(x,y)$ is finite for all $x$ and $y$. Furthermore, we obviously have $d(x,y)\leq D_{\varepsilon}(x,y) \leq D_0(x,y)$ for all $x,y \in X$ and $\varepsilon > 0$. In general, however, $D_0$ need not be finite even for a compact connected metric space. For instance, if we take the usual Koch snowflake with its subspace metric in $\mathbb{R}^2$, then $D_0(x,y)$ will be infinite for all pairs of points. Moreover, even if $D_0$ is finite and uniformly bounded, the resulting metric space $(X,D_0)$ need not be topologically equivalent to the original space $(X,d)$. If $X$ is the boundary of the unit square in $\mathbb{R}^2$ with the vertical segments from the top boundary to the bottom attached at the points $\left(\frac{1}{2^n},0\right)$, then $X$ with the subspace metric is compact and path connected, though not locally connected. In this case, $D_0(x,y)$ is finite for every $(x,y)$, but the resulting space $(X,D_0)$ is not compact. The issue that arises here is that the while $D_0$ is bounded, the relative distortion $\frac{D_0(x,y)}{d(x,y)}$ is unbounded.

\begin{proof}[Proof of Theorem \protect\ref{ulcc}]
The proof that $(X,D_{0})$ has no refinement critical values follows as in the proof of Theorem \ref{eintrinsic metric}. In fact, if $\varepsilon<\delta $ and $\{x,y\}$ is a $\delta $-chain with respect to $D_{0}$, then $D_{\varepsilon }(x,y)\leq D_{0}(x,y)<\delta $, and the preceding proof goes through without change.

For the second part, it suffices to prove that for every $x,y\in X$ there is a midpoint between $x$ and $y$ (cf. \cite{BBI} or \cite{PS}), i.e. a point $m\in X$ such that $D_{0}(x,m)=D_{0}(y,m)\leq D_{0}(x,y)/2$. By definition of $D_{0}$ for every natural number $i$ there is a $\frac{1}{i}$-chain $\alpha:=\{x=x_{0},...,x_{n}=y\}$ such that $L(\alpha )<D_{0}(x,y)+\frac{1}{i}$. Let $j_{i}$ be the largest index such that $L(\{x_{0},...,x_{j_{i}}\})\leq D_{0}(x,y)/2$. By the triangle inequality,
\begin{equation*}
L\left( \{x_{0},...,x_{j_{i}}\}\right) \geq L\left(\{x_{0},...,x_{j_{i}+1}\}\right) -\frac{1}{i}>\frac{D_{0}(x,y)}{2}-\frac{1}{i}.
\end{equation*}
Setting $m_{i}:=x_{j_{i}}$ we have by definition $\,D_{\frac{1}{i}}(x,m_{i})\leq \frac{D_{0}(x,y)}{2}$. On the other hand, 
\begin{align*}
D_{\frac{1}{i}}(y,m_{i})& \leq L\left( \{m_{i}=x_{j_{i}},...,y\}\right) =L\left( \alpha \right) -L\left( \{x_{0},...,x_{j_{i}}\}\right)  \\
&< D_{0}(x,y)+\frac{1}{i}-\left( \frac{D_{0}(x,y)}{2}-\frac{1}{i}\right)  \\
& =\frac{D_{0}(x,y)}{2}+\frac{2}{i}\text{.}
\end{align*}
Combining these inequalites we obtain 
\begin{equation*}
D_{\frac{1}{i}}(x,m_{i})\leq \frac{D_{0}(x,y)}{2}\ \text{ and }\ D_{\frac{1}{i}}(y,m_{i})\leq \frac{D_{0}(x,y)}{2}+\frac{2}{i},\ \ i\geq 1\text{. }
\end{equation*}
By choosing a subsequence if necessary we may assume that $m_{i}\rightarrow m\in X$. Since $X$ is compact, it follows from Dini's Theorem (cf. \cite{R}) that $D_{\frac{1}{i}}$ actually converges to $D_{0}$ uniformly. Thus, given any $k>0$, we have $D_{0}(x,m_{i})-\frac{1}{k}<D_{\frac{1}{i}}(x,m_{i})\leq D_{0}(x,m_{i})$ and $D_{0}(y,m_{i})-\frac{1}{k}<D_{\frac{1}{i}}(y,m_{i})\leq D_{0}(y,m_{i})$ for all large $i$. For a fixed $k$, it follows that 
\begin{equation*}
D_{0}(x,m_{i})-\frac{1}{k}<\frac{D_{0}(x,y)}{2}\ \ \ \text{and}\ \ \ D_{0}(y,m_{i})-\frac{1}{k}<\frac{D_{0}(x,y)}{2}+\frac{2}{i}
\end{equation*}
for all sufficiently large $i$. Letting $i\rightarrow \infty $ in these inequalities, we obtain 
\begin{equation*}
D_{0}(x,m)-\frac{1}{k}\leq \frac{D_{0}(x,y)}{2}\ \ \ \text{and}\ \ \ D_{0}(y,m)-\frac{1}{k}\leq \frac{D_{0}(x,y)}{2}.
\end{equation*}
Finally, we let $k\rightarrow \infty $, and this is the desired result.
\end{proof}

\vspace{.1 in} Note that an immediate corollary of the above proof is that $(X,D_{\varepsilon})$ converges in the uniform sense (cf. \cite{BBI}) - hence, in the Gromov-Hausdorff sense - to $(X,D_0)$ as $\varepsilon \rightarrow 0$.

\section{Examples of Non-Discrete Spectra}

In this section we will present several examples illustrating some of the phenomena that can occur concerning the critical spectrum of a general compact metric space. We will first prove some technical results that facilitate identifying refinement critical values. Roughly speaking, the following construction yields a method for detecting or constructing essential gaps. Indeed, the following discussion and definition should make it clear why the name `essential gap' is appropriate for the structure that induces a refinement critical value.

Let $X$ be a connected metric space. Assume there are points, $x,y\in X$, with $d(x,y)=l>0$, and a number $\varepsilon ^{\ast }>l$ such that the following holds: for each $\varepsilon $ in the interval $(l,\varepsilon^{\ast }]$, if we let $B_{x}=B(x,\varepsilon -l)$ and $B_{y}=B(y,\varepsilon - l)$, then we can express $X$ as a disjoint union, $X=Z\cup Y$, such that

\begin{enumerate}
\item[1)] $B_x \subset Z$ and $B_y \subset Y$ (hence $B_x \cap B_y =\emptyset$),\ 
\item[2)] the only points in $Z$ that are strictly within $\varepsilon $ of a point in $B_{y}$ lie in $B_{x}$, and the only points of $Y$ that are strictly within $\varepsilon $ of a point in $B_{x}$ lie in $B_{y}$.
\end{enumerate}

\noindent If these conditions hold, we call $\{x,y\}$ a \textit{pre-essential gap}. A pre-essential gap need \textit{not} be an essential gap. Note, also, that though it behaves like one locally around $x$ and $y$, $\{Z,Y\}$ need not be a disconnection of $X$. As we have already seen, connected spaces can have refinement critical values.

Given a pre-essential gap, $\{x,y\}$, with $d(x,y)=l < \varepsilon \leq \varepsilon^{\ast}$ as above, let $\gamma =\{x_{0},\dots ,x_{n}\}$ be any $\varepsilon$-chain in $X$. A pair of consecutive points, $(x_{i-1},x_{i})$, $1\leq i\leq n$, will be said to \textit{contain} or \textit{cross the $x,y$-gap} if and only if $x_{i-1}$ lies in either $B_{x}$ or $B_{y}$ and $x_{i}$ lies in the other ball. Assign each pair of consecutive points a value $|x_{i-1},x_{i}|$ as follows: 
\begin{equation}
|x_{i-1},x_{i}|= 
\begin{cases}
0, & (x_{i-1},x_{i}) \ \text{does not contain the} \ x,y\text{-gap} \\ 
1, & x_{i-1}\in B_{x}, \ x_{i}\in B_{y} \\ 
-1, & x_{i-1}\in B_{y}, \ x_{i}\in B_{x}.
\end{cases}
\notag
\end{equation}
\noindent Note that the order of the points in the notation $|x_{i-1},x_{i}|$ does matter; the second case in this definition, for instance, occurs when the first point of the pair lies in $B_{x}$ and the second point lies in $B_{y}$, while the third case occurs when the opposite holds. Now, define $\mathscr G(\gamma ;x,y,\varepsilon ):=\sum_{i=1}^{n}|x_{i-1},x_{i}|$. We call this the \textit{$(x,y,\varepsilon )$-gap number of $\gamma $}; it measures the net number of times $\gamma $ crosses the $x,y$-gap.

\begin{lemma}
\label{essential gap homotopy invariance} Assume that the above conditions hold for some $\varepsilon ^{\ast }>l=d(x,y)$, so that $\{x,y\}$ is a pre-essential gap. Given $\varepsilon $ such that $l<\varepsilon \leq \varepsilon ^{\ast }$, the integer $\mathscr G(\gamma ;x,y,\varepsilon )$ is an $\varepsilon $-homotopy invariant. That is, for fixed $\varepsilon \in (l,\varepsilon ^{\ast }]$, if $\alpha $ and $\gamma $ are $\varepsilon $-chains such that $\alpha \sim _{\varepsilon }\gamma $, then $\mathscr G(\gamma ;x,y,\varepsilon) =\mathscr G(\alpha ;x,y,\varepsilon )$.
\end{lemma}

\begin{proof}
Since any $\varepsilon $-homotopy taking $\gamma $ to $\alpha $ will consist of a finite sequence of basic moves, it suffices to prove the result in the case where $\alpha $ is obtained by adding or removing a single point to/from $\gamma $. The proof is not difficult, but it is a tedious process in working through all the possible cases. We will prove one case to illustrate the reasoning used. The rest of the cases follow in exactly the same manner.

Let $\gamma =\{x_{0},x_{1},\dots ,x_{n}\}$, and assume that $\alpha $ is obtained by adding $z$ between $x_{i-1}$ and $x_{i}$. Since this basic move only affects three different pairs of points in the sums defining the $(x,y,\varepsilon )$-gap numbers of $\gamma $ and $\alpha $, we only need to show that $|x_{i-1},x_{i}|=|x_{i-1},z|+|z,x_{i}|$.

Assume that $|x_{i-1},x_{i}|=0$. If $|x_{i-1},z|=|z,x_{i}|=0$, then the result is clear. If $|x_{i-1},z|=1$ and $|z,x_{i}|=-1$ (or $|x_{i-1},z|=-1$ and $|z,x_{i}|=1$), the result is also clear. The subcase $|x_{i-1},z|=1=|z,x_{i}|$ cannot occur, for the first equality would imply that $x_{i-1}\in B_{x}$ and $z\in B_{y}$, while the second would imply that $z\in B_{x}$ and $x_{i}\in B_{y}$, which would further imply that $z\in B_{x}\cap B_{y}$, a contradiction. The case $|x_{i-1},z|=-1=|z,x_{i}|$ also cannot occur, for the first equality would imply $x_{i-1}\in B_{y}$ and $z\in B_{x}$, while the second would imply that $z\in B_{y}$ and $x_{i}\in B_{x}$, another contradiction. Suppose $|x_{i-1},z|=1$ and $|z,x_{i}|=0$. Then $x_{i-1}\in B_{x}$, $z\in B_{y}$, and $x_{i}$ cannot be in $B_{x}$ or $B_{y}$ (or else we would have $|x_{i-1},x_{i}|=1$ in the latter case and $|z,x_{i}|=-1$ in the former). But $x_{i}$ must lie in $Z$ or $Y$, and $x_{i}$ is strictly within $\varepsilon $ of $x_{i-1}$, a point in $B_{x}$, and strictly within $\varepsilon $ of $z$, a point in $B_{y}$. If $x_{i}\in Z$, then, since $d(x_{i},z)<\varepsilon $, condition 2 above implies that $x_{i}\in B_{x}$, a contradiction. If $x_{i}\in Y$, then $d(x_{i},x_{i-1})<\varepsilon$ implies that $x_{i}\in B_{y}$, another contradiction. Similar reasoning applies to the cases $|x_{i-1},z|=-1$ and $|z,x_{i}|=0$, $|x_{i-1},z|=0$ and $|z,x_{i}|=1$, and $|x_{i-1},z|=0$ and $|z,x_{i}|=-1$. Thus, given that $|x_{i-1},x_{i}|=0$, the only possible cases that can occur result in the equality $|x_{i-1},x_{i}|=|x_{i-1},z|+|z,x_{i}|$. Proceeding, one would argue simliarly for the cases $|x_{i-1},x_i| = \pm 1$, and then work through the same procedure in the case where a point is removed from $\gamma$ to obtain $\alpha$. All cases that can occur lead to the desired equality, thus proving the result.
\end{proof}

\begin{lemma}[Essential Gap Lemma]
\label{essential gap lemma} Let $X$ be a chain connected metric space, and suppose $\{x,y\}$ is a pre-essential gap with $d(x,y)=l$. If $dist(B(x,r),B(y,r))=d(x,y)$ for all sufficiently small $r$, then $\{x,y\}$ is an essential $l$-gap.
\end{lemma}

\begin{proof}
Let $\varepsilon ^{\ast }>l$ be as in the definition of a pre-essential gap, and we may assume that $\varepsilon ^{\ast }-l$ is small enough that $dist(B(x,r),B(y,r))=d(x,y)$ for all $r\leq \varepsilon ^{\ast }-l$. Fix $\varepsilon $ so that $l<\varepsilon \leq \varepsilon ^{\ast }$. Then $\gamma :=\{x,y\}$ is an $\varepsilon $-chain, and $\mathscr G(\gamma
;x,y,\varepsilon )=1$. No $l$-chain can cross the $x,y$-gap. In fact, if $\{z_{0},\dots ,z_{n}\}$ is an $l$-chain, and if we had $z_{i-1}\in B_{x}=B(x,\varepsilon -l)$ and $z_{i-1}\in B_{y}=B(y,\varepsilon -l)$, then we would have $dist(B_{x},B_{y})\leq d(z_{i-1},z_{i})<l$, contradicting the fact that $dist(B_{x},B_{y})=d(x,y)=l$. Thus, the $(x,y,\varepsilon )$-gap number of any $l$-chain must be $0$. The $\varepsilon $-homotopy invariance of this value then implies that $\gamma=\{x,y\}$ is not $\varepsilon $-homotopic to an $l$-chain. Since $\varepsilon \in (l,\varepsilon ^{\ast })$ was arbitrary, it follows that $\{x,y\}$ is an essential $l$-gap.
\end{proof}

\vspace{0.1in} \noindent A partial converse to the Essential Gap Lemma also holds. However, this condition alone is not sufficient to ensure that $\{x,y\} $ is an essential gap.

\begin{lemma}
\label{converse egl} If $\{x,y\}$ is an essential gap in $X$, then $dist(B(x,r),B(y,r))=d(x,y)$ for sufficiently small $r$.
\end{lemma}

\begin{proof}
Suppose the conclusion does not hold. Let $l=d(x,y)$, and let $\varepsilon > l$ be given. Choose $r < \min\{l,(\varepsilon - l)/2\}$ such that there are points $u \in B(x,r)$, $v \in B(y,r)$ satisfying $dist(B(x,r),B(y,r))\leq d(u,v) < d(x,y)=l$. We can transform $\{x,y\}$ via $\varepsilon$-homotopy as follows: $\{x,y\} \rightarrow \{x,u,y\} \rightarrow \{x,u,v,y\}$. The second step is valid since $d(u,y) \leq d(u,v) + d(v,y) < l + r < (\varepsilon+l)/2 < \varepsilon$. Moreover, this last chain is an $l$-chain, since $d(u,v) < d(x,y) = l$ and $r < l$. In other words, $\{x,y\}$ can be $\varepsilon$-refined to an $l$-chain for all $\varepsilon$ sufficiently close to $l$, contradicting the hypothesis.
\end{proof}

\vspace{.1 in} This is a rather technical set-up, and it may not yet be visually clear what this structure looks like. The following can be taken as a sort of canonical example illustrating this concept.

\begin{example}
\label{essential gap example} \emph{Let $L$, $l_{1}$, $l_{2}$, and $h$ be positive real numbers such that $L$ is significantly larger than $l_{1}$ (say, $L>3l_{1}$), $l_{2}\leq l_{1}$, and $h^{2}+(l_{1}+l_{2})^{2}/4>l_{1}^{2}$. Let $X$ be the metric subspace of $\mathbb{R}^{2}$ shown in Figure \ref{fig:egl_example}, and let $x$, $y$, $u$, and $v$ be the points $\bigl(\frac{L-l_1}{2},h\bigr)$, $\bigl(\frac{L+l_1}{2},h\bigr)$, $\bigl(\frac{L-l_2}{2},0\bigr)$, and $\bigl(\frac{L+l_2}{2},0\bigr)$, respectively. Let $d$ be the diagonal distance from $x$ to $v$ (or $y$ to $u$, by symmetry).} 

\begin{figure}[hbtp]
\centering
\includegraphics[scale=1.0]{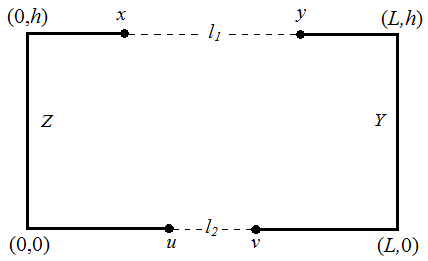}
\caption{An essential gap.}
\label{fig:egl_example}
\end{figure}

\noindent \emph{The condition $h^{2}+(l_{1}+l_{2})^{2}/4>l_{1}^{2}$ implies that $d>l_{1}$. Let $\varepsilon ^{\ast }$ be such that $l_{1}<\varepsilon ^{\ast }<\min\{d,2l_{1},(L-l_{1})/2\}$. Now, fix any $\varepsilon $ such that $l_{1}<\varepsilon \leq \varepsilon ^{\ast }$. Let $Z $ be the left half of $X $, and let $Y$ be the right half. The conditions $\varepsilon <2l_{1}$ and $\varepsilon <(L-l_{2})/2$ ensure that the balls $B_{x}:=B(x,\varepsilon -l_{1})$ and $B_{y}:=B(y,\varepsilon -l_{1})$ do not intersect or extend to the vertical sides of $X$.}

\emph{Suppose that $z\in Z$ and is strictly within $\varepsilon $ of a point in $B_{y}$. Clearly, $z$ cannot lie on the vertical segment of $Z$. If $z$ lies on the lower boundary of $Z$, then the closest it could be to any point in $B_{y}$ is $d$, which occurs when $z=u$. But $d>\varepsilon ^{\ast }\geq \varepsilon $, so $z$ cannot, in fact, lie on this lower boundary. If $z$ lies on the upper boundary of $Z$ but outside of $B_{x}$, then it is at least $\varepsilon -l_{1}+l_{1}=\varepsilon $ away from any point of $B_{y}$. Thus, it must hold that $z$ lies in $B_{x}$. Likewise, by symmetry, if $z\in Y$ and is strictly within $\varepsilon $ of a point in $B_{x}$, then $z\in B_{y}$. Moreover, we clearly have that $dist(B(x,r),B(y,r))=d(x,y)$ for sufficiently small $r$. Therefore, $\{x,y\}$ is an essential $l$-gap, and $l$ is a refinement critical value.}

\emph{To see what makes this essential gap phenomenon occur, consider the trapezoid $\{x,u,v,y\}$. The diagonals of this trapezoid are longer than the longest base of the trapezoid. This, essentially, is why $\{x,y\}$ cannot be $\varepsilon $-refined to an $l_{1}$-chain for $\varepsilon \in (l_{1},\varepsilon ^{\ast })$. Adding $v$ in between $x$ and $y$ is not
allowed, because $\varepsilon \leq \varepsilon ^{\ast }<d=d(x,v)$. Likewise, one cannot jump from $x$ to $u$ to $y$ for the same reason. In other words, because the diagonal is too long, one cannot overcome the $\{x,y\}$-gap by going around it, at least via \textquotedblleft hops\textquotedblright\ that are sufficiently close to $l_{1}$ in length. Note that if $d\leq l_{1}$ then we could, in fact, go around the $\{x,y\}$-gap. Indeed, for any $\varepsilon >l_{1}$, we could then transform the $\varepsilon $-chain, $\{x,y\}$, via $\varepsilon $-homotopy by adding $v$ and then $u$. So, it is the diagonal length that makes this gap essential.}

\emph{Finally, $X$ is not connected, but we could attach a long joining curve to $X$ to make it path connected, without affecting the critical value. $\blacksquare$}
\end{example}

Now, we will use the Essential Gap Lemma to produce examples of compact metric spaces having critical spectra with positive limit points. Moreover, these examples will show that critical values of one type can converge to critical values of the other type. They will also illustrate some of the other properties mentioned in the previous section.

\begin{example}
\label{rapunzel comb one} \emph{We define the following sets.}

\begin{enumerate}
\item[1)] \emph{For $n \geq 0$, $A_n = \{(x,y) \in \mathbb{R}^2 \: : \: 0 \leq x \leq 1,\: y = 1/2^n\} \cup \{(x,y) \in \mathbb{R}^2\: : \: 2 \leq x \leq 3, \: y=1/2^n\}$.}
\item[2)] $A_{\infty} = \{(x,y) \in \mathbb{R}^2 \: : \: 0 \leq x \leq 1, \: y=0\} \cup \{(x,y) \in \mathbb{R}^2\: : \: 2 \leq x \leq 3, \: y=0\}$.
\item[3)] $B_1 = \{(x,y) \in \mathbb{R}^2 \: : \: x=0, \: 0 \leq y \leq 2\}$, $B_2 = \{(x,y) \in \mathbb{R}^2 \: : \: x=3,\: 0 \leq y \leq 2\}$.
\item[4)] $C=\{(x,y) \in \mathbb{R}^2 \: : \: 0 \leq x \leq 3, \: y=2\}$.
\end{enumerate}

\noindent \emph{Define a metric subspace of $\mathbb{R}^2$ by $X=\bigl(\bigcup_{n=0}^{\infty} A_n \bigr) \cup A_{\infty} \cup B_1 \cup B_2 \cup C$. For $n \geq 0$, let $x_n = (1,\frac{1}{2^n})$ and $y_n = (2,\frac{1}{2^n})$, and let $x_{\infty} = (1,0)$, $y_{\infty} = (2,0)$, $z_0 = (\frac{3}{2},2)$. Let $d_0 = d(x_0,z_0)$, and, for $n\geq 1$, let $d_n = d(x_{n-1},y_n)$. Note that $d_0 = d_1$. For $m > n \geq 0$, let $d_m^n = d(x_n,y_m)$, and note that $d_n^{n-1} = d_n$ for $n\geq 1$. See Figure \ref{fig:rapcomp1} below. We call $X$ a ``Rapunzel's Comb.''}

\emph{The following results can be easily verified: 1) $1 < d_n < d_{n-1} \: \: \forall \: n \geq 1$, and $d_n \searrow 1 \: \text{as} \: n\rightarrow \infty$; 2) $d_m^n > 1$ $\forall m>n\geq 0$, and, for fixed $n$, $d_m^n$ is minimized when $m=n+1$. Now, fix $n\geq 1$. It is evident that $dist(B(x_{n},r),B(y_{n},r))=d(x_{n},y_{n})$ for sufficiently small $r$. Fix any $\varepsilon $ such that $d(x_{n},y_{n})=1<\varepsilon \leq d_{n+1}$. Let $Z$ be the left half of $X$, including $z_{0}$ (so $Z$ is closed), and let $Y$ be the rest of the space. Let $B_{x_{n}}=B(x_{n},\varepsilon -1) \subset Z$ and $B_{y_{n}}=B(y_{n},\varepsilon -1) \subset Y$. Suppose $z\in Z $ is strictly within $\varepsilon$ of a point in $B_{y_{n}}$. Clearly $z \notin B_{1}$, and $z \notin C$. In fact, the closest any point of $C\cap Z$ can be to $B_{y_{n}}$ is the distance from $z_{0}$ to $y_{n}$, which is greater than $d_{0}$. But $d_{0}=d_{1}>d_{n+1}\geq \varepsilon $, so $z \notin C$. Thus, $z$ is in the left half of one of the sets, $A_{k}$. However, $z$ cannot be in $A_{k}$ for $0\leq k<n$, since - in that case - the distance between $z$ and $B_{y_{n}}$ would be at least $d_{n}^{k}$, which, in turn, is at least as great as $d_{n}^{n-1}=d_{n}$. Since $d_{n}>d_{n+1}\geq \varepsilon $, this shows that this case cannot occur. Hence, $z$ must lie on $A_{m}$ for some $m\geq n$. If $z$ were in $A_{m}$ for $m>n$, the distance between $z$ and $B_{y_{n}}$ would be at least $d_{m}^{n}$. For fixed $n$, $d_{m}^{n}$ is minimized when $m=n+1$, so the distance between $z$ and any point of $B_{y_{n}}$ is at least $d_{n+1}^{n}=d_{n+1}\geq \varepsilon $. This contradicts that $z$ is strictly within $\varepsilon $ of a point of $B_{y_{n}}$. Therefore, $z \in A_n$, and
since it is within $\varepsilon $ of a point of $B_{y_{n}}$, it must lie in $B_{x_{n}}$. By symmetry, if $z\in Y$ and is strictly within $\varepsilon$ of a point of $B_{x_{n}}$, then $z \in B_{y_{n}}$. Hence, $\{x_{n},y_{n}\}$ is an essential $1$-gap, and $1$ is a refinement critical value.}

\begin{figure}[hbt]
\centering
\includegraphics[scale=1.0]{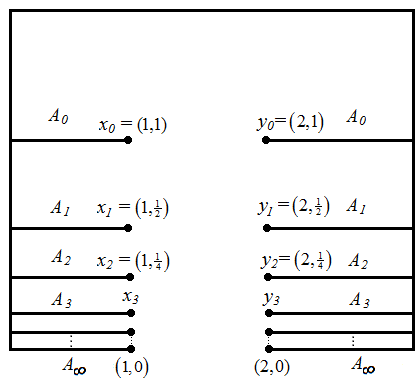}
\caption{Rapunzel's Comb}
\label{fig:rapcomp1}
\end{figure}

\emph{Fix $n\geq 1$, and let $\gamma_{n}=\{x_{n},x_{n+1},y_{n+1},y_{n},x_{n}\}$. For $1<\varepsilon \leq d_{n+1}$, $\gamma_n$ is an $\varepsilon $-loop with $(x_{n},y_{n},\varepsilon )$-gap number $-1$. An $\varepsilon $-null loop has $(x_{n},y_{n},\varepsilon )$-gap number $0$, implying that $\gamma_{n}$ is $\varepsilon $-nontrivial. This holds for all $1<\varepsilon \leq d_{n+1}$. However, for $\varepsilon >d_{n+1}$, $\gamma _{n}$ is easily seen to be trivial. Thus, $d_{n+1} \in H(X)$, and $d_{n}\searrow 1$, showing that $Cr(X) $ is not discrete. Note that $X$ is compact, path connected, and even simply connected.}

\emph{This example illustrates some other interesting properties. There are infinitely many essential $1$-gaps, but there are no non-trivial $1$-loops in $X$. Thus, $\pi _{1}(X)$ is trivial, and the map $\varphi _{\varepsilon,1}:X_{1}\rightarrow X_{\varepsilon }$ is injective for all $\varepsilon $ sufficiently close to $1$. So, there is a sequence, $d_n$, of critical values of one type converging to a critical value, $1$, of an entirely different type. $\blacksquare$}
\end{example}

There are many different variations on Rapunzel's Comb that one can use to illustrate critical value limiting behavior. All of them use the Essential Gap Lemma in some form, and the details follow much as before.

\begin{example}[Rapunzel's Comb - Variation 1]
\label{rapunzel comb two} \emph{For $n\geq 1$, let $h_{n}=2^{-n/2}$ and $H=\sum_{n=1}^{\infty }h_{n} =1+\sqrt{2}$. Let $X$ be the subspace of $\mathbb{R}^2$ shown in Figure \ref{fig:rapcomb2} below (purposefully not drawn exactly to scale to show detail). Here, the gaps increase in length to a limiting gap of length $1$. We also define the following: $z_0 = \bigl(\frac{3}{2},H+2\bigr)$, $d_0 = d(x_{\infty},z_0) = d(y_{\infty},z_0)$, $d_n =d(x_n,y_{n+1})$ $\forall$ $n \geq 1$, $x_{\infty}=(1,H)$, $y_{\infty}=(2,H)$, and}
\begin{equation*}
x_n=\Biggl(1+\frac{1}{2^{n+1}} \ , \ \sum_{i=1}^{n-1}h_i\Biggr), \: \: \ y_n = \Biggl(2- \frac{1}{2^{n+1}} \ , \ \sum_{i=1}^{n-1}h_i\Biggr), \: \forall \: n\geq 1.
\end{equation*}

\begin{figure}[hbpt]
\centering
\includegraphics[scale=1.0]{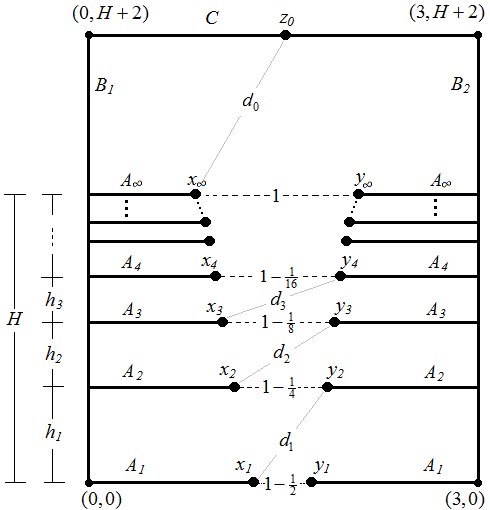}
\caption{Rapunzel's Comb - Variation 1}
\label{fig:rapcomb2}
\end{figure}



\emph{Reasoning as in the previous example, one can show that each $\{x_{n},y_{n}\} $, for $n\geq 1$, is an essential $(1-\frac{1}{2^{n}})$-gap. These values converge up to $1$, but $1$ is not a refinement critical value or even an upper non-surjective critical value. In fact, for every $\varepsilon >1$, every $\varepsilon $-chain can be $\varepsilon $-refined to a $1$-chain, even the chain $\{x_{\infty },y_{\infty }\}$. Given any $\varepsilon > 1$, the distances $d(x_{\infty},y_n)$ and $d(y_{\infty},x_n)$ are eventually less than $\varepsilon$, thus making this refinement possible. So, what type of critical value is $1$?}

\emph{Fix $n\geq 2$, and let $\gamma _{n}$ be the loop $\{x_{n},x_{n-1},y_{n-1},y_{n},x_{n}\}$. For all $\varepsilon $ greater than $1-\frac{1}{2^{n}}$, $\gamma _{n}$ is an $\varepsilon $-loop. Since $\{x_{n},y_{n}\}$ is an essential gap, we also know that - for each $\varepsilon $ greater than but sufficiently close to $1-\frac{1}{2^{n}}$ - the $(x_{n},y_{n},\varepsilon )$-gap number is an $\varepsilon $-homotopy invariant. Fixing any such $\varepsilon $, we see that the $\varepsilon$-chain $\alpha_{n}:=\{x_{n},y_{n}\}$ has non-zero $(x_{n},y_{n},\varepsilon) $-gap number, while the $\varepsilon $-chain $\beta_{n}:=\{x_{n},x_{n-1},y_{n-1},y_{n}\}$ does not cross the $x_{n},y_{n}$-gap at all. Hence, $\gamma _{n}=\beta _{n}\alpha _{n}^{-1}$ cannot be $\varepsilon $-null for such an $\varepsilon $. So, $\gamma _{n}$ is $\varepsilon $-nontrivial for all $\varepsilon $ greater than but sufficiently close to $1-\frac{1}{2^{n}}$. On the other hand, $\gamma _{n}$ is $1$-null. In fact, since the diagonals between $x_{n}$ and $y_{n-1}$ are less than $1$ in length (this is easily verified), we can successively remove $y_{n-1}$, $x_{n-1}$, and $y_{n}$ from $\gamma _{n}$, giving us the trivial chain. Therefore, $1$ is a lower non-injective critical value that is an upper limit of sequences of refinement critical values and homotopy critical values. Note that, as before, $X$ is compact, path connected, and simply-connected. $\blacksquare$}
\end{example}

\begin{example}[Rapunzel's Comb - Variation 2]
\label{rapunzel comb three} \emph{The construction of this example is very similar to the previous case. In fact, the lengths of the gaps will be the same. The key difference will be changing the heights between the gaps in the comb. We increase them just enough so that the diagonal lengths between $x_n$ and $y_{n+1}$ are greater than $1$ for each $n$, but still small enough so that the sum of the heights is finite.}

\emph{So, for $n \geq 1$, let $h_n = \frac{\sqrt{3}}{(\sqrt{2})^n}$, and let $H = \sum_{n=1}^{\infty} h_n =\sqrt{3}+\sqrt{6}$. We define $X$ exactly as in the previous example except for the different values $h_n$, and we similarly define the points $x_n$, $y_n$, $x_{\infty}$, $y_{\infty}$, and $z_0$. Also as before, we let $d_n = d(x_n,y_{n+1})$, so that $d_n$ is the length of the diagonal between $x_n$ and $y_{n+1}$. Since the construction is the same, Figure \ref{fig:rapcomb2} holds equally well for this example. We just need to keep in mind that the heights, $h_n$, and, therefore, the diagonals, $d_n$, are larger in this case. One can verify by direct computation the following:} 
\begin{equation*}
d_n^2 = 1 + \frac{3}{2^{n+1}} + \frac{9}{2^{2n+4}} \: \: \: \text{and} \: \: \: d_{n+1} < d_n \: \forall \: n \: \Rightarrow \: d_n \searrow 1,
\end{equation*}
\begin{equation*}
d(x_n,y_m) > 1 \: \: \text{for all} \: \: 1 \leq n < m.
\end{equation*}
\noindent \emph{In addition, for fixed $n$ and $m > n$, the diagonal lengths, $d(x_n,y_m)$, increase as $m$ increases.}

\emph{Now, fix $n\geq 2$, and, recalling that $d(x_{n},y_{n})=1-\frac{1}{2^{n}}$, let $\varepsilon $ be such that} 
\begin{equation}
1-\frac{1}{2^{n}}<\varepsilon \leq \min \Bigl\{d(x_{1},y_{n}),\dots,d(x_{n-1},y_{n}),d(x_{n+1},y_{n}),1+h_{n}-\frac{1}{2^{n}}\Bigr\}. 
\notag
\end{equation}
\noindent \emph{The condition that $\varepsilon $ be less than or equal to $1+h_{n}-\frac{1}{2^{n}}$ is to ensure that the ball of radius $\varepsilon-(1-\frac{1}{2^{n}})$ centered at $x_{n}$ (or $y_{n}$) does not intersect any nearby teeth of the comb or either of the vertical sides of $X$. That is, these balls are just segments of the teeth of the comb formed by $A_{n}$. As before, we let $Z$ be the left half of $X$, and we let $Y$ be the right half. Suppose $z\in Z$ and lies within $\varepsilon$ of a point of $B_{y_{n}}:=B(y_{n},\varepsilon -(1-\frac{1}{2^{n}}))$. Clearly, $z$ cannot lie in $C$ or $B_{1}$. If $z$ were in $A_{m}$ for some $m\leq n-1$, the distance between $z$ and any point of $B_{y_{n}}$ would be at least $d(x_{m},y_{n})$. But $\varepsilon \leq d(x_{m},y_{n})$ for such $m$, so this cannot occur. If $z$ were in $A_{n+1}$, then the closest $z$ could be to any point of $B_{y_{n}}$ is $d_{n}=d(x_{n+1},y_{n})$, but, again, we have $\varepsilon \leq d(x_{n+1},y_{n})$. So, this cannot occur either. Niether can $z$ be in $A_{m}$ for $m>n+1$, since the diagonal lengths, $d(x_{m},y_{n})$, are greater than $d_{n}$ for $m>n+1$. Hence, $z$ must lie in $A_{n}$, and, in fact, it must lie in $B(x_{n},\varepsilon -(1-\frac{1}{2^{n}}))$. By symmetry, the same result holds if $z\in Y$ and lies within $\varepsilon$ of a point in $B(x_{n},\varepsilon -(1-\frac{1}{2^{n}}))$. We also have $dist(B(x_{n},r),B(y_{n},r))=d(x_{n},y_{n})$ for sufficiently small $r$, so $\{x_{n},y_{n}\}$ is an essential $(1-\frac{1}{2^{n}})$-gap. It follows that $1-\frac{1}{2^{n}}=d(x_{n},y_{n})$ is an upper non-surjective critical value; for all $\varepsilon $ greater than but sufficiently close to $1-\frac{1}{2^{n}}$, $\{x_{n},y_{n}\}$ is an $\varepsilon $-chain that cannot be $\varepsilon $-refined to a $(1-\frac{1}{2^{n}})$-chain.}

\emph{Finally, since $1-\frac{1}{2^{n}}\nearrow 1$, we know that $1$ is a critical value. Note that} 
\begin{equation*}
1<\min \Bigl\{d(x_{1},y_{n}),\dots,d(x_{n-1},y_{n}),d(x_{n+1},y_{n}),1+h_{n}-\frac{1}{2^{n}}\Bigr\},
\end{equation*}
\noindent \emph{because all diagonals have length greater than $1$ and} 
\begin{equation*}
h_{n}>\frac{1}{(\sqrt{2})^{n}}>\frac{1}{2^{n}}\Rightarrow 1+h_{n}-\frac{1}{2^{n}}>1.
\end{equation*}
\noindent \emph{Thus, for $\varepsilon =1$, the $(x_{n},y_{n},\varepsilon )$-gap number of an $\varepsilon $-chain is an $\varepsilon $-homotopy invariant. Now, $\{x_{n},y_{n}\}$ is a $1$-chain, and its $(x_{n},y_{n},1)$-gap number is $1$. However, no $(1-\frac{1}{2^{n}})$-chain can cross the $x_{n},y_{n}$-gap. Thus, $\{x_{n},y_{n}\}$ cannot be $1$-homotopic to a $(1-\frac{1}{2^{n} })$-chain. In other words, the map $\varphi_{1,1-1/2^{n}}:X_{1-1/2^{n}}\rightarrow X_{1}$ is not surjective, and this holds for all $n\geq 1$. Hence, $1$ is a lower non-surjective critical value. As we have shown, such a critical value can only occur as the upper limit of upper non-surjective critical values. $\blacksquare$}
\end{example}

\begin{example}[Rapunzel's Comb - Variation 3]
\label{rapunzel comb four} \emph{Working as in the previous examples, let $X$ be the variation of Rapunzel's Comb shown in Figure \ref{fig:rapcomb3}. Here, we have a sequence of gaps of length $1+\frac{1}{2^{n}}$ converging down to a gap of length $1$. Moreover, we have attached an extra single set of teeth below the limiting gap, which adds a gap of length $l < 1$.}

\emph{We choose the heights, $h_{n}$, to be $h_{n}=\frac{1}{(\sqrt{2})^n}$, and we choose $h$ so that $d(x_{\infty},b)=d(y_{\infty},a)=1$ and $d(x_{\infty},a)$, $d(y_{\infty},b) < 1$. A straightforward computation shows that this can be done. It can then be shown as in the previous examples that $\{x_n,y_n\}$ is an essential $(1+\frac{1}{2^n})$-gap for each $n$. Thus, for each $n$ and all $\varepsilon$ greater than but sufficiently close to $1+\frac{1}{2^n}$, $\{x_n,y_n\}$ cannot be $\varepsilon$-refined to a $(1+\frac{1}{2^n})$-chain and, thus, to a $1$-chain, either. Hence, $1$ is an upper non-surjective critical value. However, there is no essential $1$-gap in $X$. Because we have added the gap $\{a,b\}$ below the gap, $\{x_{\infty},y_{\infty }\}$, and chosen $h$ so that $d(x_{\infty},b)=d(y_{\infty },a)=1$, the chain $\{x_{\infty },y_{\infty }\}$ can be $\varepsilon $-refined to a $1$-chain for all $\varepsilon $ greater than $1$. The homotopy, itself, is simply} 
\begin{equation*}
\{x_{\infty },y_{\infty }\}\rightarrow \{x_{\infty},a,y_{\infty}\}\rightarrow \{x_{\infty },a,b,y_{\infty }\}.
\end{equation*}
\noindent \emph{Since $l<1$, this is a $1$-chain. This example shows that - even in a compact space - an upper non-surjective critical value need not correspond to an essential gap. As we have already shown, the two only correspond for certain when the critical value in question is isolated. $\blacksquare$}
\end{example}

\begin{figure}[bph]
\centering
\includegraphics{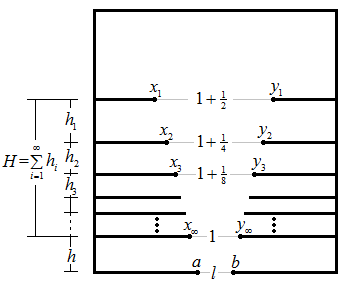}
\caption{Rapunzel's Comb - Variation 3}
\label{fig:rapcomb3}
\end{figure}

\addtocontents{toc}{\protect\setcounter{tocdepth}{0}}
\subsection*{Acknowledgements}
\addtocontents{toc}{\protect\setcounter{tocdepth}{1}}
This research was partly supported by the Tennessee Science Alliance and NSF Grant DMS-0552774.

\end{document}